\documentclass[11pt,openany]{article}
\usepackage{amsmath}
\usepackage{amsthm}
\usepackage{amscd}
\usepackage{amssymb}
\usepackage{tikz}
\usepackage{latexsym}
\usepackage{lmodern}
\usepackage[all]{xy}
\usetikzlibrary{matrix,arrows}
\newcommand{\cw}{{\curlywedge}}
\newcommand{\F}{{\mathbb F}}
\newcommand{\C}{{\mathbb C}}

\newcommand{\f}{\frac}
\newcommand{\lo}{\longrightarrow}

\renewcommand{\Im}{{\rm Im}}
\newtheorem{theorem}{Theorem}[section]

\newtheorem{corollary}[theorem]{Corollary}

\newtheorem{lemma}[theorem]{Lemma}

\newtheorem{proposition}[theorem]{Proposition}

\theoremstyle{definition}
\newcommand{\ds}{\displaystyle}

\newtheorem{example}[theorem]{Example}

\title{\bf The exterior product and homology of Hom-Lie algebras}
\author{Negur Shahni Karamzadeh , Seyedeh Narges Hosseini\footnote{
Corresponding author.} , Ali Reza Salemkar\\
{\small Department of Mathematics, Faculty of Mathematical
Sciences, Shahid Beheshti University, Tehran, Iran
}\\
{\small n$_{-}$shahni@sbu.ac.ir~,~narges.hosseini90@gmail.com~,~
salemkar@sbu.ac.ir}}
\date{ }
\begin{document}
\maketitle
\begin{abstract}
In this article, we use the theory of $($non-abelian$)$ exterior
product of Hom-Lie algebras to prove the Hopf's formula for these
algebras. As an application, we construct an eight-term sequence
in the homology of Hom-Lie algebras. We also investigate the
capability property of Hom-Lie algebras via the exterior product.
\\[.2cm]
\end{abstract}

\section{Introduction}
In \cite{E,E1}, Ellis introduced the notions of the non-abelian
tensor and exterior products of Lie algebras and gave some of
their basic properties. He investigated their relations to the low
dimensional homology of Lie algebras. In particular, he proved
that for any Lie algebra $L$ over a field $\F$, $H_2(L)$, the
second homology of $L$ with coefficients in the trivial $L$-module
$\F$, is isomorphic to the kernel of the commutator map
$\lambda_L:L\wedge L\lo L$. Moreover, he showed how the tensor
product is related to the universal central extensions. In
\cite{NPR}, Niroomand et al. dealt with some of the applications
of exterior product to the notion of capability of Lie algebras.
There is a series of papers (for instance, see
\cite{ASR,EHS,IKL,K,K1,NJP,SE,SEA}) emphasizing the relevance of
these notions to the development and exposition of the basic
theories of the capability and the second homology of Lie
algebras.\vspace{.2cm}

Hom-Lie algebras were originally introduced in \cite{HLS}, to
construct deformations of the Witt algebra (which is the Lie
algebra of derivations on the Laurent polynomial algebra
$\C[x,x^{-1}]$), and the Virasoro algebra (which is a complex Lie
algebra defined as the unique central extension of the Witt
algebra), and used in \cite{LS} to study quantum deformations and
discretisation of vector fields via twisted derivations. Hom-Lie
algebras are also useful tools in studying mathematical physics.
In this regard, Yau (\cite{Y1,Y2}) gave some applications of these
algebras to a generalization of the Yang-Barter equation and to
braid group representations. Hom-Lie algebras are $\F$-vector
spaces endowed with a bilinear skew-symmetric bracket satisfying a
Jacobi identity twisted by a map. When this map is identity, then
the definition of Lie algebras is recovered. Hence, one is
motivated  to investigate some appropriate results of the Lie
algebra theory in the general setting of the category of Hom-Lie
algebras. Accordingly, several concepts of Lie algebras have been
generalized to Hom-Lie algebras. Hom-Lie algebras on semisimple
Lie algebras in \cite{JL}, representation theory and (co)homology
theory in \cite{AEM, CS,S,Y}, universal central extensions  in
\cite{CIP} and finally, the theory of tensor products of Hom-Lie
algebras in \cite{CKR}, are studied. Recently, Casas and
Garcia-Martinez (\cite{CG1}) introduced the notions of
(non-abelian) exterior product and capability of Hom-Lie algebras
and extended some classical results from the Schur multiplier of
Lie algebras to the multiplier of Hom-Lie algebras. In this paper,
we continue the same line of research of Casas and
Garcia-Martinez. In fact, we obtain further structure properties
of the exterior products of Hom-Lie algebras, and use them to give
several homology results, generalizing known ones for Lie
algebras. In particular, we prove the Hopf's formula for Hom-Lie
algebras. Also, the relationship between the exterior product and
the capability is studied.\vspace{.15cm}

This paper is organized as follows: In Section 2, we review some
basic definitions regarding Hom-Lie algebras, with a special
mention of the Hom-action and the homology of Hom-Lie algebras.
Also, in this section we give some necessary results for the
subsequent use in this paper. In Section 3, we recall and gather
some general results related to the tensor and exterior products
of Hom-Lie algebras. In addition, we provide some relations
between these products and a generalized version of Whitehead's
universal quadratic functor for Hom-vector spaces. In Section 4,
several key results from the theory of homology of Lie algebras
are extended to Hom-Lie algebras. In particular, we describe the
second homologies of Hom-Lie algebras as central subalgebras of
their exterior products, and obtain an exact sequence of eight
terms associated with an extension of Hom-Lie algebras in the
homology. In the last section, we show that to investigate the
capability of non-perfect Hom-Lie algebras, we only need to study
those Hom-Lie algebras whose companion endomorphisms are
surjective. We also establish the capability property for Hom-Lie
algebras which are perfect or abelian. Furthermore, a relation
between the tensor
and exterior centers of a Hom-Lie algebra is given.\vspace{.2cm}\\
{\large{\bf Notations}}. Throughout this paper, all vector spaces
and algebras are considered over some fixed field $\F$ and linear
maps are $\F$-linear maps. We write $\otimes$ and $\wedge$ for the
usual tensor and exterior products of vector spaces over $\F$,
respectively. For any vector space (resp., Hom-Lie algebra) $L$, a
subspace (resp., an ideal) $L'$ and $x\in L$, we write $\bar x$ to
denote the coset $x+L'$. If $A$ and $B$ are subspaces of a vector
space $V$ for which $V=A+B$ and $A\cap B=0$, we will write
$V=A\dot{+}B$.
\section{Preliminaries on Hom-Lie algebras}
In this section, we recall some basic concepts and auxiliary
facts, which will be needed in the sequel.
\subsection{Basic definitions}
A {\it Hom-Lie algebra} $(L,\alpha_L)$ is a non-associative
algebra $L$ together with a linear map $\alpha_L:L\longrightarrow
L$ satisfying the conditions

$(i)$
$[x,y]=-[y,x]$,\hspace{10.35cm}(skew-symmetry)\vspace{-.05cm}

$(ii)$
$[\alpha_L(x),[y,z]]+[\alpha_L(y),[z,x]]+[\alpha_L(z),[x,y]]=0$,\hspace{4cm}
(Hom-Jacobi identity)\\ for all $x,y,z\in L$, where $[~,~]$
denotes the product in $L$. In the whole paper we only deal with
(the so-called {\it multiplicative}) Hom-Lie algebras
$(L,\alpha_L)$ such that $\alpha_L$ preserves the product, that
is, $\alpha_L([x,y])=[\alpha_L(x),\alpha_L(y)]$ for all $x,y\in
L$. The Hom-Lie algebra $(L,\alpha_L)$ is said to be {\it regular}
if $\alpha_L$ is bijective. Taking $\alpha_L=id_L$, we recover
exactly Lie algebras. A {\it Hom-vector space} is a pair
$(V,\alpha_V)$, where $V$ is a vector space and $\alpha_V:V\lo V$
is a linear map. If we put $[x,y]=0$ for all $x,y\in V$, then
$(V,\alpha_V)$ is a Hom-Lie algebra, which is called an {\it
abelian Hom-Lie algebra}.

A {\it homomorphism} of Hom-Lie algebras
$\delta:(L_1,\alpha_{L_1})\longrightarrow(L_2,\alpha_{L_2})$ is an
algebra homomorphism from $L_1$ to $L_2$ such that $\delta
\circ\alpha_{L_1}=\alpha_{L_2}\circ\delta$. The corresponding
category of Hom-Lie algebras is denoted by $\bf{HomLie}$. As this
category is a variety of $\Omega$-groups in the sense of Higgins
\cite{H}, it is a semi-abelian category. Therefore, the
$3\times3$-Lemma and the Snake Lemma hold in this category
\cite{BB}. We refer the reader to \cite{CG} for obtaining more
information on this category.

Let $(L,\alpha_L)$ be a Hom-Lie algebra. The following is a list
of the concepts which will be used:

$(i)$ A (Hom-Lie) {\it subalgebra} $(H,\alpha_H)$ of
$(L,\alpha_L)$ consists of a vector subspace $H$ of $L$, which is
closed under the product and invariant under the map $\alpha_L$,
together with the linear self-map $\alpha_H$ being the restriction
of $\alpha_L$ on $H$. In such a case we may write $\alpha_L|$ for
$\alpha_H$. A subalgebra $(H,\alpha_H)$ is an {\it ideal} if
$[x,y]\in H$ for all $x\in H$, $y\in L$.

$(ii)$ The {\it center} of $(L,\alpha_L)$ is the vector space
$Z(L)=\{x\in L~|~[x,y]=0$, for all $y\in L\}$.\\
Put $Z_{\alpha}(L)=\{x\in L~|~[\alpha_L^k(x),y]=0,$ for all $y\in
L, k\geq0\}$, where $\alpha_L^0=id_L$ and $\alpha_L^k$, $k\geq1$,
denotes the composition of $\alpha_L$ with itself $k$ times. It is
straightforward to see that
$(Z_{\alpha}(L),\alpha_{Z_{\alpha}(L)})$ is the largest central
ideal of $(L,\alpha_L)$. We call $Z_{\alpha}(L)$ the $\alpha$-{\it
center} of $(L,\alpha_L)$. When $\alpha_L$ is a surjective
homomorphism or $(L,\alpha_L)$ is abelian, then
$Z_{\alpha}(L)=Z(L)$.

$(iii)$ If $(H,\alpha_H)$ and $(K,\alpha_K)$ are two ideals of
$(L,\alpha_L)$, then the {\it$($Higgins$)$ commutator} of
$(H,\alpha_H)$ and $(K,\alpha_K)$, denoted by
$([H,K],\alpha_{[H,K]})$, is a Hom-Lie subalgebra spanned by the
elements $[h,k]$, $h\in H$, $k\in K$. Note that
$([H,K],\alpha_{[H,K]})$ is an ideal of $(H,\alpha_H)$ and
$(K,\alpha_K)$. Especially, $([L,L],\alpha_{[L,L]})$ is an ideal
of $(L,\alpha_L)$. The quotient $(L/[L,L],\bar\alpha_L)$ is called
the {\it abelianisation} of $(L,\alpha_L)$, and denoted by
$(L^{ab},\alpha_{L^{ab}})$. The Hom-Lie algebra $(L,\alpha_L)$ is
called {\it perfect} if $L=[L,L]$.

$(iv)$ An exact sequence
$(M,\alpha_M)\rightarrowtail(K,\alpha_K)\twoheadrightarrow(L,\alpha_L)$
of Hom-Lie algebras is a {\it central extension} of $(L,\alpha_L)$
if $M\subseteq Z(K)$, or equivalently, $[M,K]=0$.\vspace{.2cm}

Let $(M,\alpha_M)$ and $(N,\alpha_N)$ be two Hom-Lie algebras. By
a {\it Hom-action} of $(M,\alpha_M)$ on $(N,\alpha_N)$, we mean a
$\F$-bilinear map $M\times N\to N,
(m,n)\mapsto\hspace{-.1cm}~^mn$, satisfying the axioms

$(a)$$~^{[m,m']}\alpha_N(n)=~^{\alpha_M(m)}(~^{m'}n)
-\hspace{-.1cm}~^{\alpha_M(m')}(~^{m}n)$,

$(b)$$~^{\alpha_M(m)}[n,n']=[~^mn,\alpha_N(n')]+[\alpha_N(n),\hspace{-.1cm}~^mn']$,

$(c)$ $\alpha_N(~^mn)=~^{\alpha_M(m)}\alpha_N(n)$,\\
for all $m,m'\in M, n,n'\in N$. The action is called {\it trivial}
if $~^mn=0$, for all $m\in M$, $n\in N$. Also, if $(N,\alpha_N)$
is an abelian Hom-Lie algebra enriched with a Hom-action of
$(M,\alpha_M)$, then $(N,\alpha_N)$ is said to be a {\it
Hom-$M$-module} (see \cite{Y}). Clearly, if $(M,\alpha_M)$ is a
subalgebra of some Hom-Lie algebra $(L,\alpha_L)$ and
$(N,\alpha_N)$ is an ideal of $(L,\alpha_L)$, then the product in
$L$ induces a Hom-action of $(M,\alpha_M)$ on $(N,\alpha_N)$ given
by$~^mn=[m,n]$. In particular, there is a Hom-action of
$(L,\alpha_L)$ on itself given by the product in $L$.\vspace{.2cm}

Let $(M,\alpha_M)$ and $(L,\alpha_L)$ be Hom-Lie algebras together
with a Hom-action of $(L,\alpha_L)$ on $(M,\alpha_M)$. Their {\it
semi-direct product} $(M\rtimes L,\alpha_{\rtimes})$ is the
Hom-Lie algebra with the underlying vector space $M\dot{+}L$, the
endomorphism $\alpha_{\rtimes}:M\rtimes L\lo M\rtimes L$ given by
$\alpha_{\rtimes}(m,l)=(\alpha_M(m),\alpha_L(l))$, and the
product\vspace{-.15cm}
\[[(m_1,l_1),(m_2,l_2)]=([m_1,m_2]+\hspace{-.17cm}~^{\alpha_L(l_1)}m_2
+\hspace{-.17cm}~^{\alpha_L(l_2)}m_1, [l_1,l_2]).\vspace{-.13cm}\]
When $(L,\alpha_L)$ acts trivially on $(M,\alpha_M)$, we get the
direct sum structure of Hom-Lie algebras.

Let $(M,\alpha_M)\stackrel{i}\rightarrowtail(K,\alpha_K)
\stackrel{\zeta}\twoheadrightarrow(L,\alpha_L)$ be a split short
exact sequence of Hom-Lie algebras, that is, there exists a
homomorphism $\eta:(L,\alpha_L)\lo(K,\alpha_K)$ such that
$\zeta\circ\eta=id_L$. Then we can find a Hom-action of
$(L,\alpha_L)$ on $(M,\alpha_M)$ defined by
$~^lm=i^{-1}[\eta(l),i(m)]$ for all $m\in M$, $l\in L$.
Furthermore, we have the following commutative diagram with exact
rows:\vspace{.2cm}

\tikzset{node distance=3cm, auto}$~~~~$
\begin{tikzpicture}[%
back line/.style={densely dotted}, cross
line/.style={preaction={draw=white, -,line width=6pt}}]
\hspace{1.2cm}
\node(A1){\fontsize{9.5}{5}\selectfont$~$};
\node[right of=A1](B1){\fontsize{9.5}{5}\selectfont$(M,\alpha_M)$};
\node[right of=B1](C1){\fontsize{9.5}{5}\selectfont$(M\rtimes L,\alpha_{\rtimes})$};
\node[right of=C1](D1){\fontsize{9.5}{5}\selectfont$(L,\alpha_L)$};
\node(B2)[below of=B1, node distance=1.7cm]{\fontsize{9.5}{5}\selectfont$(M,\alpha_{M})$};
\node(C2)[below of=C1, node distance=1.7cm]{\fontsize{9.5}{5}\selectfont$(K,\alpha_K)$};
\node(D2)[below of=D1, node distance=1.7cm]{\fontsize{9.5}{5}\selectfont$(L,\alpha_{L}),$};

\draw[>->](B1) to node{$i$}(C1);
\draw[>->>](B1) to node[right]{\fontsize{9.5}{5}\selectfont$id_M$}(B2);
\draw[>->](B2) to node{\fontsize{9.5}{5}\selectfont$\subseteq$}(C2);
\draw[>->>](C1) to node[right]{\fontsize{9.5}{5}\selectfont$\xi$}(C2);
\draw[->>](C1) to node{\fontsize{9.5}{5}\selectfont$\rho$}(D1);
\draw[->>](C2) to node{\fontsize{9.5}{5}\selectfont$\zeta$}(D2);
\draw[>->>](D1) to node[right]{\fontsize{9.5}{5}\selectfont$id_L$}(D2);
\end{tikzpicture}\\
where $i(m)=(m,0)$, $\rho(m,l)=l$, and $\xi$ is an isomorphism
defined by $\xi(m,l)=m+\eta(l)$. In particular, if we put
$T=\eta(L)$, then $(T,\alpha_T)$ is a subalgebra of $(K,\alpha_K)$
such that $K=M\dot{+}T$.

A {\it crossed module} is a homomorphism
$\partial:(M,\alpha_M)\to(L,\alpha_L)$ of Hom-Lie algebras
together with a Hom-action of $(L,\alpha_L)$ on $(M,\alpha_M)$
such that $\partial(\hspace{-.1cm}~^lm)=[l,\partial(m)]$ and
$^{\partial(m)}{m'} =[m,m']$ for all $m,m'\in M$, $l\in L$. If
$(M,\alpha_M)$ is an ideal of $(L,\alpha_L)$, then the inclusion
map $(M,\alpha_M)\hookrightarrow(L,\alpha_L)$ is a crossed module.
It is worth noting that for any crossed module
$\partial:(M,\alpha_M)\to(L,\alpha_L)$,
$(\Im(\partial),{\alpha_L}_|)$ is an ideal of $(L,\alpha_L)$ and
$(\ker(\partial),{\alpha_M}_|)$ is a central subalgebra of
$(M,\alpha_M)$.
\subsection{The Schur multiplier and the homology of Hom-Lie algebras}
Consider the functors
$\mathfrak{U}:\bf{HomLie}\longrightarrow\bf{HomSet}$, from the
category of Hom-Lie algebras to the category of Hom-sets, that
assigns to any Hom-Lie algebra $(B,\alpha_B)$ the Hom-set obtained
by forgetting the operations, and
$\mathfrak{F_r}:\bf{HomSet}\longrightarrow\bf{HomLie}$, that
assigns to any Hom-set $(X,\alpha_X)$ the free Hom-Lie algebra
$\mathfrak{F_r}(X,\alpha_X)=(A_X/I,\bar\alpha_A)$ (here,
$(A_X,\alpha_A)$ is a non-associative Hom-algebra such that $A_X$
is the $\F$-algebra generated by the free magma $M_X$ and
$\alpha_A$ is the endomorphism induced by $\alpha_X$; and $I$ is
the two-side ideal of $(A_X,\alpha_A)$ spanned by the elements of
the forms\vspace{-.2cm}
\[~~~~~ab+ba~~~~{\rm and}~~~~
\alpha_A(a)(bc)+\alpha_A(b)(ca)+\alpha_A(c)(ab),\vspace{-.2cm}\]
where $a,b,c\in A_X$). It is proved in \cite{CG} that the functor
$\mathfrak{F_r}$ is the left adjoint to the functor
$\mathfrak{U}$. As an immediate consequence, we conclude that
every Hom-Lie algebra admits at least one free presentation.

Let $(L,\alpha_L)$ be an arbitrary Hom-Lie algebra with a free
presentation $(R,\alpha_R)\stackrel{\subseteq}{\rightarrowtail}(
F,\alpha_F)\stackrel{\rho}{\twoheadrightarrow}(L,\alpha_L)$. Then
the {\it Schur multiplier} of $(L,\alpha_L)$ is defined in
\cite{CG1} to be the abelian Hom-Lie algebra\vspace{-.05cm}
\[{\cal{M}}(L,\alpha_L)=\frac{(R,\alpha_R)\cap
([F,F],\alpha_{[F,F]})}{([F,R],\alpha_{[F,R]})}.\vspace{-.05cm}\]
As the category of Hom-Lie algebras is semi-abelian, Theorem 6.9
of \cite{EV} indicates that the Schur multiplier of $(L,\alpha_L)$
is independent of the choice of the free presentation of
$(L,\alpha_L)$.\vspace{.2cm}

The homology of Hom-Lie algebras, which is a generalization of the
Chevalley-Eilenberg homology of Lie algebras, is constructed as
follows: Let $(L,\alpha_L)$ be a Hom-Lie algebra and
$(M,\alpha_M)$ be a Hom-$L$-module. The homology of $(L,\alpha_L)$
with coefficients in $(M,\alpha_M)$, denoted by
$H_{\ast}^{\alpha}(L,M)$, is the homology of the Hom-chain complex
$(C_{\ast}^{\alpha}(L,M),d_{\ast})$, where
$C_{n}^{\alpha}(L,M)=(M\otimes L^{\wedge
n},\alpha_M\otimes\alpha_L^{\wedge n})$, $n\geq0$ ($L^{\wedge n}$
denotes the $n$-th exterior power of $L$, with $L^{\wedge0}=\F$),
and the boundary map $d_n:C_{n}^{\alpha}(L,M)\lo
C_{n-1}^{\alpha}(L,M)$, $n\geq1$, is a homomorphism of Hom-vector
spaces defined by
\begin{align*}
d_n&(m\otimes x_1\otimes\dots\otimes x_n)=\sum_{i=1}^n(-1)^i
(\hspace{-.1cm}~^{x_i}m\otimes\alpha_L(x_1)\wedge
\dots\wedge\widehat{\alpha_L(x_i)}\wedge\dots\wedge\alpha_L(x_n))\\
&+\sum_{1\leq i<j\leq
n}(-1)^{i+j}\alpha_M(m)\otimes[x_i,x_j]\wedge
\alpha_L(x_1)\wedge\dots\wedge\widehat{\alpha_L(x_i)}
\wedge\dots\wedge\widehat{\alpha_L(x_j)}
\wedge\dots\wedge\alpha_L(x_n),
\end{align*}
where the notation $\widehat{\alpha_L(x_i)}$ means that the
variable $\alpha_L(x_i)$ is omitted. It is obvious that
$H_{n}^{\alpha}(L,M)=(\ker(d_n)/\Im(d_{n-1}),\overline{\alpha_M\otimes\alpha_L^{\wedge
n}})$ has a Hom-vector space structure for all $n\geq1$ (note that
$H_{n}^{\alpha}(L,M)$ is defined in \cite{Y} to be the vector
space $\ker(d_n)/\Im(d_{n-1})$, while we here take it as a
Hom-vector space). In the special case, if
$(M,\alpha_M)=(\F,id_{\F})$ is a trivial Hom-$L$-module, then
$H_n^\alpha(L,\F)$ is said to the {\it $n$-th homology} of
$(L,\alpha_L)$ and denoted by $H_n^\alpha(L)$. It is easily
checked that there is an isomorphism of Hom-vector spaces
$H_1^\alpha(L)\cong (L^{ab},\alpha_{L^{ab}})$. Also, it is
established in \cite[Theorem 3.14]{CG1} that
$H^{\alpha}_2(L)\cong{\cal M}(L,\alpha_L)$ for any perfect Hom-Lie
algebra $(L,\alpha_L)$. In Section 4, we extend this result for
any arbitrary Hom-Lie algebra.
\section{The tensor and the exterior products of Hom-Lie algebras}
This section is devoted to the study of the properties of the
(non-abelian) tensor and exterior products of Hom-Lie algebras. We
begin by recalling these concepts.

Let $\partial_1:(M,\alpha_M)\lo(L,\alpha_L)$ and
$\partial_2:(N,\alpha_N)\lo(L,\alpha_L)$ be two crossed modules of
Hom-Lie algebras. There are Hom-actions of $(M,\alpha_M)$ on
$(N,\alpha_N)$ and of $(N,\alpha_N)$ on $(M,\alpha_M)$ given by
$~^mn=\hspace{-.17cm}~^{\partial_1(m)}n$ and $~^nm=
\hspace{-.17cm}~^{\partial_2(n)}m$. We take $(M,\alpha_M)$ (and
$(N,\alpha_N)$) to act on itself by the product. Then the {\it
$($Hom-Lie$)$ tensor product} $(M\star N,\alpha_{\star})$ is
defined in \cite{CKR} as the Hom-Lie algebra generated by the
symbols $m\star n$ (for $m\in M, n\in M$) subject to the
relations\vspace{-.18cm}
\begin{alignat*}{2}
&(A1)~~c(m\star n)=cm\star n=m\star cn,~~~~~
&&(A4)~~[m,m']\star\alpha_N(n)=\alpha_M(m)\star(^{m'}n)-\alpha_M(m')\star(^mn),\\[-.2cm]
&(A2)~~(m+m')\star n=m\star n+m'\star n,~~~~~
&&(A5)~~\alpha_M(m)\star[n,n']=(^{n'}m)\star\alpha_N(n)-(^nm)\star\alpha_N(n'),\\[-.2cm]
&(A3)~~m\star(n+n')=m\star n+m\star n',~~~~~
&&(A6)~~[(m\star n),(m'\star n')]=-(^nm)\star(^{m'}n'),
\end{alignat*}

\vspace{-.2cm}\hspace{-.57cm}for all $c\in\F$, $m,m'\in M$,
$n,n'\in N$, and the endomorphism $\alpha_{\star}$ is given on
generators by\vspace{-.18cm}
\[\alpha_{\star}(m\star
n)=\alpha_M(m)\star\alpha_N(n).\vspace{-.18cm}\] Note that the
identity homomorphism $id_L:(L,\alpha_L)\lo (L,\alpha_L)$ is a
crossed module with $(L,\alpha_L)$ acting on itself by the
product, so we can always form the tensor products $(L\star
M,\alpha_{\star})$, $(L\star N,\alpha_{\star})$ and $(L\star
L,\alpha_{\star})$. Also, if $\alpha_M=id_M$, $\alpha_N=id_N$ and
$\alpha_L=id_L$, then $M\star N$ coincides with the tensor product
of Lie algebras given in \cite{E1}.

The following proposition gives some useful information on Hom-Lie
tensor products, the proof of which are left to the reader (see
also \cite{CKR}).
\begin{proposition}
With the above assumptions and notations, we have $:$

$(i)$ The maps\vspace{-.4cm}
\begin{alignat*}{1}
\lambda&:(M\star N,\alpha_{\star})\lo(L,\alpha_L)~~~~~~
,~~~~~m\star
n\longmapsto[\partial_1(m),\partial_2(n)]\\[-.17cm]
\lambda_M&:(M\star
N,\alpha_{\star})\lo(M,\alpha_M)~~~~,~~~~~m\star n\longmapsto
-\hspace{-.1cm}~^nm\\[-.17cm]
\lambda_N&:(M\star
N,\alpha_{\star})\lo(N,\alpha_N)~~~~~,~~~~~m\star
n\longmapsto\hspace{-.1cm}~^mn
\end{alignat*}
are homomorphisms of Hom-Lie algebras with the kernels contained
in the center of $(M\star N,\alpha_{\star})$.

$(ii)$ There is a Hom-action of $(L,\alpha_L)$ on $(M\star
N,\alpha_{\star})$ given by\vspace{-.22cm}
\[~^l(m\star n)=\hspace{-.1cm}~^lm\star\alpha_N(n)+\alpha_M(m)
\star\hspace{-.1cm}~^ln,\vspace{-.22cm}\]and then $(M,\alpha_M)$
and $(N,\alpha_N)$ act on $(M\star N,\alpha_{\star})$ via
$\partial_1$ and $\partial_2$. Moreover,\vspace{-.2cm}
\begin{equation}
\lambda(\hspace{-.1cm}~^lx)=[\alpha_L(l),\lambda(x)]~~~~~ {\rm
and}~~~~~\hspace{-.1cm}~^{\lambda(x)}(x')=[\alpha_{\star}(x),x']\vspace{-.2cm}
\end{equation}
for all $x,x'\in M\star N$, $l\in L$, and the relations similar to
$(1)$ are valid for $\lambda_M$ and $\lambda_N$.

$(iii)$ If $(M,\alpha_M)$ and $(N,\alpha_N)$ act trivially on each
other and both maps $\alpha_M$ and $\alpha_N$ are surjective, then
there is an isomorphism of abelian Hom-Lie algebras $(M\star
N,\alpha_{\star})\cong(M^{ab}\otimes N^{ab},\alpha_{\otimes})$,
where $\alpha_{\otimes}$ is induced by $\alpha_M$ and $\alpha_N$.

$(iv)$ Let
$(M,\alpha_M)\rightarrowtail(K,\alpha_K)\twoheadrightarrow(L,\alpha_L)$
and
$(M',\alpha_{M'})\rightarrowtail(K',\alpha_{K'})\twoheadrightarrow(L',\alpha_{L'})$
be short exact sequences of Hom-Lie algebras, where
$(M',\alpha_{M'})$ and $(K',\alpha_{K'})$ are ideals of
$(K,\alpha_K)$, and $(L',\alpha_{L'})$ is an ideal of
$(L,\alpha_L)$. Then there exists an exact sequence of Hom-Lie
algebras\vspace{-.2cm}
\[((M\star K')\rtimes(K\star M'),\alpha_{\rtimes})\lo(K\star K',\alpha_{\star})
\twoheadrightarrow(L\star L',\alpha_{\star}),\vspace{-.2cm}\]
where the Hom-action of $(K\star M',\alpha_{\star})$ on $(M\star
K',\alpha_{\star})$ is induced by the homomorphism
$\lambda_{M'}:(K\star M',\alpha_{\star})\lo(M',\alpha_{M'})$.
\end{proposition}
In view of part $(i)$ of the above proposition, for any Hom-Lie
algebra $(L,\alpha_L)$, the Hom-map $\lambda_L:(L\star
L,\alpha_{\star})\lo(L,\alpha_L)$, $l_l\star
l_2\longmapsto[l_1,l_2]$, is a homomorphism (which is called the
{\it commutator Hom-map}). We define
$J_2^{\alpha}(L)=(\ker\lambda_L,\alpha_{\star}|)$.\vspace{.2cm}

Let $(M\square N,\alpha_{\square})$ be the Hom-vector subspace of
$(M\star N,\alpha_{\star})$, where $M\square N$ is the vector
subspace of $M\star N$ spanned by the elements of form $m\star n$
with $\partial_1(m)=\partial_2(n)$ and $\alpha_{\square}$ is the
restriction of $\alpha_{\star}$ to $M\square N$. Then $M\square
N\subseteq Z(M\star N)$ and so $(M\square N,\alpha_{\square})$ is
an ideal of $(M\star N,\alpha_{\star})$. Following \cite{CG1}, the
{\it $($Hom-Lie$)$ exterior product} $(M\cw N,\alpha_{\cw})$ is
defined to be the quotient $((M\star N)/(M\square
N),\bar\alpha_{\star})$. We write $m\cw n$ to denote the image in
$M\cw N$ of the generator $m\star n$.

It is readily seen that the parts $(i)$ and $(ii)$ of Proposition
3.1 hold with $\star$ replaced by $\cw$. We can also conclude from
Proposition 3.1$(iv)$ that every short exact sequence $e:
(M,\alpha_M)\rightarrowtail(K,\alpha_K)\twoheadrightarrow(L,\alpha_L)$
of Hom-Lie algebras induces an exact sequence\vspace{-.2cm}
\begin{equation}
(M\cw K,\alpha_{\cw})\lo(K\cw K,\alpha_{\cw})
\twoheadrightarrow(L\cw L,\alpha_{\cw}).\vspace{-.2cm}
\end{equation}
Note that the functorial homomorphisms $(K\cw
M,\alpha_{\cw})\lo(K\cw K,\alpha_{\cw})$ and $(M\cw
K,\alpha_{\cw})\lo(K\cw K,\alpha_{\cw})$ have the same images. We
can say more if the extension $e$ is split.
\begin{lemma}
A split extension
$(M,\alpha_M)\rightarrowtail(K,\alpha_K)\twoheadrightarrow(L,\alpha_L)$
of Hom-Lie algebras induces a short exact sequence of Hom-Lie
algebras $ (M\cw K,\alpha_{\cw})\stackrel{i}\rightarrowtail(K\cw
K,\alpha_{\cw}) \twoheadrightarrow(L\cw L,\alpha_{\cw})$.
\end{lemma}
\begin{proof}
We only require to prove that the homomorphism $i$ is injective.
To do this, it suffices to construct a homomorphism of Hom-Lie
algebras $\theta:(K\cw K,\alpha_{\cw})\lo((M\cw K)\rtimes(L\cw
L),\alpha_{\rtimes})$ such that the composite map $\theta oi$ is
the canonical inclusion. Here the Hom-action of $(L\cw
L,\alpha_\cw)$ on $(M\cw K,\alpha_\cw)$ is defined as follows:
\vspace{-.2cm}
\[~^x(m\cw k)=[\lambda_L(x),m]\cw\alpha_K(k)+\alpha_M(m)
\cw[\lambda_L(x),k],\vspace{-.1cm}\] for all $x\in L\cw L$, $m\in
M$, $k\in K$, in which $\lambda_L:(L\cw
L,\alpha_{\cw})\lo(L,\alpha_L)$, $(l_1\cw
l_2)\longmapsto[l_1,l_2]$, is the commutator Hom-map and
$(L,\alpha_L)$ is considered as a subalgebra of $(K,\alpha_K)$.
Since $(K,\alpha_K)\cong(M\rtimes L,\alpha_{\rtimes})$, we can
define\vspace{-.2cm}
\begin{alignat*}{1}
\theta:((M\rtimes L)\cw(M\rtimes L),\alpha_{\cw})&\lo((M\cw(M\rtimes L))
\rtimes(L\cw L),\alpha_{\rtimes}).\\
(m_1,l_1)\cw(m_2,l_2)&\longmapsto(m_1\cw(m_2,l_2)-m_2\cw(0,l_1),l_1\cw
l_2)
\end{alignat*}

\vspace{-.2cm}\hspace{-.57cm}It is routine to check that $\theta$
preserves the defining relations of the Hom-Lie exterior product
with $\theta\circ\alpha_{\cw}=\alpha_{\rtimes}\circ\theta$, and is
therefore the required homomorphism.
\end{proof}
To show how the exterior product is related to universal central
extensions of Hom-Lie algebras, we need the following
\begin{lemma}
If $e: (M,\alpha_M)\stackrel{\iota}\rightarrowtail(K,\alpha_K)
\stackrel{\phi}\twoheadrightarrow(L,\alpha_L)$ is a central
extension of Hom-Lie algebras, then there exists a homomorphism of
Hom-Lie algebras $\psi:(L\cw L,\alpha_{\cw})\lo(K,\alpha_K)$ such
that $\phi\circ\psi=id_L\circ\lambda_L$. Moreover, if
$(L,\alpha_L)$ is perfect then $\psi$ is unique.
\end{lemma}
\begin{proof}
We define the Hom-map $\psi$ on generators by $\psi(l_1\cw
l_2)=[k_1,k_2]$, where $\phi(k_i)=l_i$, $i=1,2$. Due to the
centrality of the extension $e$, $\psi$ preserves the relations of
the exterior product with $\psi
\circ\alpha_{\cw}=\alpha_K\circ\psi$, and is thus the required
homomorphism. If $\psi,\psi':(L\cw L,\alpha_{\cw
})\lo(K,\alpha_K)$ are two homomorphisms with $\phi\circ\psi=\phi
\circ\psi'$, then $\psi-\psi'=\iota\circ\eta$, where $\eta:(L\cw
L,\alpha_{\cw})\lo(M,\alpha_M)$ is a homomorphism such that $[L\cw
L,L\cw L]$ is contained in $\ker\eta$. By the relation (A6), if
$(L,\alpha_L)$ is perfect, then so is $(L\cw L,\alpha_{\cw })$,
forcing the uniqueness of $\psi$, as desired.
\end{proof}
The above lemma, together with \cite[Theorem 3.4]{CKR}, leads us
to the following result.
\begin{proposition}
For any perfect Hom-Lie algebra $(L,\alpha_L)$, the
extension\vspace{-.18cm}
\[(\ker(\lambda_L),\alpha_{\cw }|)\rightarrowtail(L\cw
L,\alpha_{\cw
})\stackrel{\lambda_L}\twoheadrightarrow(L,\alpha_L)\vspace{-.18cm}\]
is the universal central extension of $(L,\alpha_L)$ and so, there
is an isomorphism of Hom-Lie algebras $(L\cw
L,\alpha_{\cw})\cong(L\star L,\alpha_{\star})$. In particular,
$H_2^{\alpha}(L)\cong(\ker(\lambda_L),\alpha_\cw|)$.
\end{proposition}
We end this section by analyzing the kernel of the natural
homomorphism $(M\star N,\alpha_{\star})\twoheadrightarrow(M\cw
N,\alpha_{\cw})$. To do this, we require the following definition,
which is a generalized version of Whitehead's quadratic functor in
the context of Hom-vector spaces.\\
{\large{\bf Definition}}. The {\it universal quadratic Hom-functor
$\Gamma$} is defined for any Hom-vector space $(A,\alpha_A)$ to be
the Hom-vector space $(\Gamma(A),\alpha_{\Gamma})$, where
$\Gamma(A)$ is the vector space generated by the symbols
$\gamma(a)$ with $a\in A$, subject to the relations\vspace{-.4cm}
\begin{alignat*}{1}
\lambda^2\gamma(a)&=~\gamma(\lambda a),\\
\gamma(\lambda
a+b)+\lambda\gamma(a)+\lambda\gamma(b)&=~\lambda\gamma(a+b)
+\gamma(\lambda a)+\gamma(b),\\
\gamma(a+b+c)+\gamma(a)+\gamma(b)+\gamma(c)
&=\gamma(a+b)+\gamma(a+c)+\gamma(b+c),
\end{alignat*}
for all $\lambda\in\F$, $a,b,c\in A$, and the linear map
$\alpha_{\Gamma}:\Gamma(A)\lo\Gamma(A)$ is given by
$\alpha_{\Gamma}(\gamma(a))=\gamma(\alpha_A(a))$.\vspace{.2cm}

Under the assumptions at the beginning of this section, we
set\vspace{-.2cm}
\[M\times_LN=\{(m,n)\in M\oplus
N~|~\partial_1(m)=\partial_2(n)\}~~{\rm and}~~\langle
M,N\rangle=\{(\lambda_M(x),\lambda_N(x))~|~x\in M\star
N\}.\vspace{-.2cm}\] Then $(M\times_LN,\alpha_{\oplus})$ is a
subalgebra of the direct sum $(M\oplus N,\alpha_{\oplus})$,
$(\langle M,N\rangle,\alpha_{\oplus}|)$ is an ideal of
$(M\times_LN,\alpha_{\oplus})$, and the quotient
$(M\times_LN/\langle M,N\rangle,\bar\alpha_{\oplus})$ is abelian.
By arguments similar to those used in \cite[Proposition 14]{E1},
we can obtain the following natural exact sequence of Hom-Lie
algebras\vspace{-.07cm}
\[(\Gamma(\f{M\times_LN}{\langle
M,N\rangle}),\alpha_{\Gamma})\stackrel{\psi}\lo(M\star
N,\alpha_{\star})\twoheadrightarrow(M\cw
N,\alpha_{\cw}),\vspace{-.07cm}\]where
$\psi(\gamma(\overline{(m,n)}))=m\star n$ for $\overline{(m,n)}$
the coset of $\langle M,N\rangle$ represented by $(m,n)\in
M\times_LN$. In particular, if $(M,\alpha_M)$ and $(N,\alpha_N)$
are ideals of a Hom-Lie algebra, then there is an exact
sequence\vspace{-.07cm}
\begin{equation}
(\Gamma(\f{M\cap
N}{[M,N]}),\alpha_{\Gamma})\stackrel{\psi}\lo(M\star
N,\alpha_{\star})\twoheadrightarrow(M\cw
N,\alpha_{\cw}).\vspace{-.07cm}
\end{equation}
In next result, we show that in the case of
$(M,\alpha_M)=(N,\alpha_N)$, the homomorphism $\psi$ in $(3)$ is
injective.
\begin{proposition}
For any Hom-Lie algebra $(L,\alpha_L)$ with surjective
endomorphism $\alpha_L$, there is an exact sequence of Hom-Lie
algebras
$(\Gamma(L^{ab}),\alpha_{\Gamma})\stackrel{\psi}\rightarrowtail
(L\star L,\alpha_{\star})\twoheadrightarrow(L\cw L,\alpha_{\cw})$.
\end{proposition}
\begin{proof}
In view of the exact sequence $(3)$, it suffices to prove that
$\psi$ is injective. Owing to Proposition 3.1$(iii)$,
$(L^{ab}\star L^{ab},\alpha_{\star})\cong(L^{ab}\otimes
L^{ab},\alpha_{\otimes})$. Taking into account \cite[Proposition
8.6]{ST}, one easily sees that the composite
homomorphism\vspace{-.2cm}
\[\varphi:~(\Gamma(L^{ab}),\alpha_{\Gamma})\stackrel{\psi}\lo
(L\square L,\alpha_{\square})\stackrel{\subseteq}\lo(L\star
L,\alpha_{\star})\stackrel{nat.}\lo(L^{ab}\star
L^{ab},\alpha_{\star})\stackrel{\cong}\lo(L^{ab}\otimes
L^{ab},\alpha_{\otimes}),\vspace{-.2cm}\] maps a basis of
$\Gamma(L^{ab})$ injectively into a set of linearly independent
elements. Therefore $\varphi$ and then $\psi$ are injective.
\end{proof}
We deduce from Proposition $3.5$ the following consequence, which
is used in Section $5$.
\begin{corollary}
With the assumptions of Proposition $3.5$, there is an isomorphism
of Hom-Lie algebras $ \rho: (L\square
L,\alpha_{\square})\lo(L^{ab}\square
L^{ab},\bar\alpha_{\square})$, given by $(x\star
y)\longmapsto(\bar x\star\bar y)$.
\end{corollary}
\begin{proof}
It is sufficient to note that $\rho$ is equal to the composite
homomorphism\vspace{-.2cm}
\[\hspace{1cm}(L\square
L,\alpha_{\square})\stackrel{\cong}\lo(\Gamma(L^{ab}),
\alpha_{\Gamma})\stackrel{\cong}\lo(L^{ab}\square
L^{ab},\bar\alpha_{\square}).\vspace{-.95cm}\]
\end{proof}
\section{The Hopf's formula for Hom-Lie algebras}
In \cite{E}, Ellis proves that, for any Lie algebra $K$, the
second homology of $K$, $H_2(K)$, is isomorphic to the kernel of
the commutator map $K\wedge K\stackrel{[~,~]}\lo K$ (here $\wedge$
denotes the non-abelian exterior product of Lie algebras). Using
this result, he determines in \cite{E1} the behavior of the
functor $H_2(-)$ with respect to the direct sum of Lie algebras
and also, applying topological techniques, gets an eight-term
exact sequence in homology of Lie algebras\vspace{-.2cm}
\begin{equation}
H_3(K)\lo H_3(L)\lo H_2(K,M)\lo H_2(K)\lo H_2(L)\lo M/[M,K]\lo
H_1(K)\twoheadrightarrow H_1(L).\vspace{-.2cm}
\end{equation}
from a short exact sequence of Lie algebras $M\rightarrowtail
K\twoheadrightarrow L$ (here, $H_2(K,M)$ denotes the second
relative Chevalley-Eilenberg homology of the pair $(K,M)$, which
is isomorphic to $\ker(M\wedge K\lo K)$). In particular, he
obtains the Hopf's formula for Lie algebras and, moreover, shows
that if $F/R$ is a free presentation of $K$ and $S/R$ is the
induced presentation of $M$ for some ideal $S$ of $F$, then
$H_3(K)\cong\ker(R\wedge F\lo F)$ and $H_3(L)\cong\ker(S\wedge
F\lo F)$. In this section, we generalize these results to Hom-Lie
algebras. We start with the following theorem.
\begin{theorem}
Let $(L,\alpha_L)$ be any Hom-Lie algebra. Then there is an
isomorphism of Hom-vector spaces $ H_2^\alpha(L)
\cong(\ker(\lambda_L),\alpha_{\ker(\lambda_L)})$, where
$\lambda_L:(L\cw L,\alpha_{\cw})\lo(L,\alpha_L)$ is the commutator
Hom-map.
\end{theorem}
\begin{proof}
We recall that the linear Hom-maps $d_2:C_2^\alpha(L,\F)\to
C_1^\alpha(L,\F)$ and $d_3:C_3^\alpha(L,\F)\to C_2^\alpha(L,\F)$
are defined by $d_2(x\wedge y)=[x,y]$ and $d_3(x\wedge y\wedge
z)=-[x,y]\wedge
\alpha_L(z)+\alpha_L(x)\wedge[y,z]+[x,z]\wedge\alpha_L(y)$,
respectively. Hence ${\rm Im}(d_3)=span\{-[x,y]\wedge
\alpha_L(z)+\alpha_L(x)\wedge[y,z]+[x,z]\wedge\alpha_L(y)~|~x,y,z\in
L\}$. Obviously, the linear Hom-map $\psi:((L\wedge
L)/\Im(d_3),\overline{\alpha_L^{\wedge2}})\to(L\cw
L,\alpha_{\cw})$ given by $\psi(x\wedge y +\Im(d_3))=x\cw y$ is
surjective. We now have the following diagram of Hom-vector
spaces, in which the rows are exact:

\tikzset{node distance=3cm, auto}$~~~~$
\begin{tikzpicture}[%
back line/.style={densely dotted}, cross
line/.style={preaction={draw=white, -,line width=6pt}}]
\hspace{1.6cm}\node(A1){\fontsize{9.5}{5}\selectfont$~$};
\node[right
of=A1](B1){\fontsize{9.5}{5}\selectfont$H_2^\alpha(L)$};
\node[right of=B1](C1){\fontsize{9.5}{5}\selectfont
$(\ds\frac{L\wedge L}{{\rm
Im}(d_3)},\overline{\alpha_L^{\wedge2}})$}; \node[right
of=C1](D1){\fontsize{9.5}{5}\selectfont$([L,L],\alpha_{[L,L]})$};
\node (B2)[below of=B1, node distance=1.7cm]{\fontsize{9.5}{5}
\selectfont$\hspace{-.3cm}(\ker(\lambda_L),\alpha_{\ker(\lambda_L)})$};
\node (C2)[below of=C1, node
distance=1.7cm]{\fontsize{9.5}{5}\selectfont$(L\cw
L,\alpha_{\cw})$}; \node (D2)[below of=D1, node
distance=1.7cm]{\fontsize{9.5}{5}\selectfont$([L,L],\alpha_{[L,L]})$};

\draw[>->](B1) to node{$\subseteq$}(C1);
\draw[->>](B1) to node[right]{\fontsize{9.5}{5}\selectfont$\psi_{|}$}(B2);
\draw[>->](B2) to node{\fontsize{9.5}{5}\selectfont$~$}(C2);
\draw[->>](C1) to node[right]{\fontsize{9.5}{5}\selectfont$\psi$}(C2);
\draw[->>](C1) to node{\fontsize{9.5}{5}\selectfont$\bar d_2$}(D1);
\draw[->>](C2) to node{\fontsize{9.5}{5}\selectfont$\lambda_L$}(D2);
\draw[>->>](D1) to node[right]{\fontsize{9.5}{5}\selectfont$=$}(D2);
\end{tikzpicture}\\
Note that we here consider $(L\cw L,\alpha_{\cw})$ as a Hom-vector
space. By comparing the relation (A4) and the generators of
$\Im(d_3)$, it follows that $\psi$ has an inverse $\psi'$ that
sends $x\cw y$ to $x\wedge y +\Im(d_3)$. We therefore deduce that
$\psi_|$ is an isomorphism, as required.
\end{proof}
The following corollary is a direct consequence of the above
theorem and Proposition 3.5.
\begin{corollary}
For any Hom-Lie algebra $(L,\alpha_L)$ with surjective
endomorphism $\alpha_L$, there is an exact sequence of Hom-vector
spaces $(\Gamma(L^{ab}),\alpha_{\Gamma})\rightarrowtail
J_2^{\alpha}(L)\twoheadrightarrow H_2^{\alpha}(L)$. In particular,
if $L$ is of finite dimension, then
$\dim(J_2^{\alpha}(L))=\dim(H_2^{\alpha}(L))+\dim(\Gamma(L^{ab}))$
$($as vector spaces$)$.
\end{corollary}
\begin{theorem}
Let $(F,\alpha_F)$ be any free Hom-Lie algebra. Then there is an
isomorphism of Hom-Lie algebras $(F\cw
F,\alpha_{\cw})\cong([F,F],\alpha_{[F,F]})$.
\end{theorem}
\begin{proof}
We have only to prove that $\lambda_F:(F\cw
F,\alpha_{\cw})\lo([F,F],\alpha_{[F,F]})$ is injective. Using the
same notations as in Subsection 2.2, suppose
$(F,\alpha_F)=(A_X/I,\bar\alpha_A)$ and $(Y,\alpha_Y)$ is a
subalgebra of $(A_X,\alpha_A)$, where $Y$ is the
$\alpha_A$-invariant subspace of $A_X$ generated by the set
$\{ab~|~a,b\in A_X\}$. Note that each element $x\in Y$ is written
as a unique finite sum $\ds\sum_{i=1}^{n}a_ib_i$, $a_i,b_i\in
A_X$. Therefore, $\phi:(Y,\alpha_Y)\lo(F\cw F,\alpha_{\cw})$,
given by $ab\longmapsto\bar a\cw\bar b$, is a well-defined linear
Hom-map (where $\bar a$ denotes the coset $a+I\in A_X/I$). For any
$a,b,c\in A_X$, we have\vspace{-.1cm}
\[\phi(ab+ba)=\bar a\cw\bar b+\bar b\cw\bar a=0,\vspace{-.14cm}\]
\[\phi(\alpha_A(a)bc+\alpha_A(b)ca+\alpha_A(c)ab)=
\bar\alpha_A(\bar a)\cw[\bar b,\bar c]+\bar\alpha_A(\bar
b)\cw[\bar c,\bar a]+\bar\alpha_A(\bar c)\cw[\bar a,\bar
b])=0.\vspace{-.08cm}\] Hence $\phi$ induces a linear Hom-map
$\bar\phi:([F,F]=Y/I,\alpha_{[F,F]})\lo(F\cw F,\alpha_{\cw})$.
Furthermore, $\lambda_F\circ\bar\phi=id_{[F,F]}$ and $\bar\phi
\circ\lambda_F=id_{F\cw F}$. This completes the proof.
\end{proof}
From the above theorem, we have the following corollary.
\begin{corollary}
Let $(R,\alpha_R)\rightarrowtail(F,\alpha_F)
\stackrel{\pi}\twoheadrightarrow(L,\alpha_L)$ be a free
presentation of the Hom-Lie algebra $(L,\alpha_L)$. Then there is
an isomorphism of Hom-Lie algebras $(L\cw
L,\alpha_{\cw})\cong([F,F]/[R,F],\bar\alpha_{[F,F]})$. In
particular, ${\cal
M}(L,\alpha_L)\cong(\ker(\lambda_L),\alpha_{\ker(\lambda_L)})$.
\end{corollary}
\begin{proof}
Consider the following commutative diagram of Hom-Lie algebras
with exacts rows:\vspace{.1cm}

\tikzset{node distance=3cm, auto}$~~~~$
\begin{tikzpicture}[%
back line/.style={densely dotted}, cross
line/.style={preaction={draw=white, -,line width=6pt}}]
\hspace{1.6cm}\node(A1){\fontsize{9.5}{5}\selectfont$~$};
   \node[right of=A1](B1){\fontsize{9.5}{5}\selectfont$(F\cw R,\alpha_{\cw})$};
   \node[right of=B1](C1){\fontsize{9.5}{5}\selectfont$(F\cw F,\alpha_{\cw})$};
   \node[right of=C1](D1){\fontsize{9.5}{5}\selectfont$(L\cw L,\alpha_{\cw})$};
   \node (B2)[below of=B1, node distance=1.7cm]{\fontsize{9.5}{5}\selectfont$(\ker(\pi_|),\alpha_{\ker(\pi_|)})$};
   \node (C2)[below of=C1, node distance=1.7cm]{\fontsize{9.5}{5}\selectfont$([F,F],\alpha_{[F,F]})$};
   \node (D2)[below of=D1, node distance=1.7cm]{\fontsize{9.5}{5}\selectfont$([L,L],\alpha_{[L,L]}),$};

   \draw[->](B1) to node{$\theta$}(C1);
   \draw[->>](B1) to node[right]{\fontsize{9.5}{5}\selectfont${\lambda_F|}$}(B2);
   \draw[>->](B2) to node{\fontsize{9.5}{5}\selectfont$~$}(C2);
   \draw[->](C1) to node[right]{\fontsize{9.5}{5}\selectfont$\lambda_F$}(C2);
   \draw[->>](C1) to node{\fontsize{9.5}{5}\selectfont$~$}(D1);
   \draw[->>](C2) to node{\fontsize{9.5}{5}\selectfont$\pi_|$}(D2);
   \draw[->>](D1) to node[right]{\fontsize{9.5}{5}\selectfont$\lambda_L$}(D2);
\end{tikzpicture}\\
Evidently, $\lambda_F$ maps the subalgebra
$(\Im(\theta),\alpha_{\Im(\theta)})$ isomorphically onto
$([R,F],\alpha_{[R,F]})$. We consequently conclude from Theorem
4.3 that\vspace{-.2cm}
\[(L\cw L,\alpha_{\cw})\cong(\ds\f{F\cw F}{\Im(\theta)},\bar\alpha_{\cw})
\cong(\ds\f{[F,F]}{[R,F]},\bar\alpha_{[F,F]}),\vspace{-.2cm}\]and
the proof is complete.
\end{proof}
Combining the above corollary with Theorem $4.1$, we have the main
result of this section.
\begin{corollary}[Hopf's formula for Hom-Lie algebras]
For any Hom-Lie algebra $(L,\alpha_L)$, there is an isomorphism of
Hom-vector spaces $H_2^\alpha(L)\cong{\cal M}(L,\alpha_L)$.
\end{corollary}
As an immediate consequence of Corollary 4.5, we conclude that the
second homology of any free Hom-Lie algebra is trivial.

Using Theorem 4.1, we generalize the exact sequence $(4)$ for
Hom-Lie algebras.
\begin{theorem}
Let $e: (M,\alpha_M)\rightarrowtail(K,\alpha_K)
\twoheadrightarrow(L,\alpha_L)$ be an extension of Hom-Lie
algebras. Let $(R,\alpha_R)\rightarrowtail(F,\alpha_F)
\twoheadrightarrow(K,\alpha_K)$ be a free presentation of
$(K,\alpha_K)$ and $(M,\alpha_M)\cong(S/R,\bar\alpha_S)$ for some
ideal $(S,\alpha_S)$ of $(F,\alpha_F)$. Then there is an exact
sequence of Hom-vector spaces\vspace{-.1cm}
\[
(\ker(R\cw F\lo F),\alpha_{\cw}|)\lo(\ker(S\cw F\lo
F),\alpha_{\cw}|)\lo(\ker(M\cw K\lo K),\alpha_{\cw}|)\lo
H_2^{\alpha}(K)\vspace{-.08cm}\]
\begin{equation}
\hspace{1.2cm}\lo H_2^{\alpha}(L)\lo(\f{M}{[M,K]},\bar\alpha_M)\lo
H_1^{\alpha}(K)\twoheadrightarrow H_1^{\alpha}(L).\vspace{-.1cm}
\end{equation}
Moreover, if the extension $e$ is split, then the sequence $(5)$
induces a short exact sequence\vspace{-.12cm}
\begin{equation}
(\ker(M\cw K\lo K),\alpha_{\cw}|)\rightarrowtail
H_2^{\alpha}(K)\twoheadrightarrow H_2^{\alpha}(L).\vspace{-.2cm}
\end{equation}
\end{theorem}
\begin{proof}
Consider the following commutative diagrams of Hom-Lie
algebras:\vspace{.3cm}

\tikzset{node distance=2.4cm, auto}$~~~~$
\begin{tikzpicture}[%
back line/.style={densely dotted}, cross
line/.style={preaction={draw=white, -,line width=6pt}}]
\hspace{-2.2cm}\node(A1){\fontsize{9.5}{5}\selectfont$~$};
  \node[right of=A1](B1){\fontsize{9.5}{5}\selectfont$(M\cw K,\alpha_{\cw})$};
  \node[right of=B1](C1){\fontsize{9.5}{5}\selectfont$(K\cw K,\alpha_{\cw})$};
  \node[right of=C1](D1){\fontsize{9.5}{5}\selectfont$(L\cw L,\alpha_{\cw})$};
  \node (B2)[below of=B1, node distance=1.7cm]{\fontsize{9.5}{5}\selectfont$(M,\alpha_{M})$};
  \node (C2)[below of=C1, node distance=1.7cm]{\fontsize{9.5}{5}\selectfont$(K,\alpha_{K})$};
  \node (D2)[below of=D1, node distance=1.7cm]{\fontsize{9.5}{5}\selectfont$(L,\alpha_{L})$};

  \draw[->](B1) to node{$~$}(C1);
  \draw[->](B1) to node[right]{\fontsize{9.5}{5}\selectfont$\lambda_M$}(B2);
  \draw[>->](B2) to node{\fontsize{9.5}{5}\selectfont$~$}(C2);
  \draw[->](C1) to node[right]{\fontsize{9.5}{5}\selectfont$\lambda_K$}(C2);
  \draw[->>](C1) to node{\fontsize{9.5}{5}\selectfont$$}(D1);
  \draw[->>](C2) to node{\fontsize{9.5}{5}\selectfont$~$}(D2);
  \draw[->](D1) to node[right]{\fontsize{9.5}{5}\selectfont$\lambda_L$}(D2);
\end{tikzpicture}\\\vspace{-2.3cm}

\tikzset{node distance=2.1cm, auto} $~~~~$
\begin{tikzpicture}[%
back line/.style={densely dotted},
cross line/.style={preaction={draw=white, -,line width=4pt}}]
\hspace{6.7cm}
\node(A1){and};

\end{tikzpicture}\\\vspace{-1.96cm}

\tikzset{node distance=2.4cm, auto}$~~~~$
\begin{tikzpicture}[%
back line/.style={densely dotted}, cross
line/.style={preaction={draw=white, -,line width=6pt}}]
\hspace{6.5cm}\node(A1){\fontsize{9.5}{5}\selectfont$~$};
   \node[right of=A1](B1){\fontsize{9.5}{5}\selectfont$(R\cw F,\alpha_{\cw})$};
   \node[right of=B1](C1){\fontsize{9.5}{5}\selectfont$(S\cw F,\alpha_{\cw})$};
   \node[right of=C1](D1){\fontsize{9.5}{5}\selectfont$(M\cw L,\alpha_{\cw})$};
   \node (B2)[below of=B1, node distance=1.7cm]{\fontsize{9.5}{5}\selectfont$(R,\alpha_{R})$};
   \node (C2)[below of=C1, node distance=1.7cm]{\fontsize{9.5}{5}\selectfont$(S,\alpha_{S})$};
   \node (D2)[below of=D1, node distance=1.7cm]{\fontsize{9.5}{5}\selectfont$(M,\alpha_{M})$};

   \draw[->](B1) to node{$~$}(C1);
   \draw[->](B1) to node[right]{\fontsize{9.5}{5}\selectfont$\lambda_R$}(B2);
   \draw[>->](B2) to node{\fontsize{9.5}{5}\selectfont$~$}(C2);
   \draw[->](C1) to node[right]{\fontsize{9.5}{5}\selectfont$\lambda_S$}(C2);
   \draw[->>](C1) to node{\fontsize{9.5}{5}\selectfont$$}(D1);
   \draw[->>](C2) to node{\fontsize{9.5}{5}\selectfont$$}(D2);
   \draw[->](D1) to node[right]{\fontsize{9.5}{5}\selectfont$\lambda_M$}(D2);
\end{tikzpicture}

\hspace{-.65cm}where, the rows are exact. Applying the Snake Lemma
to these diagrams, we have the exact sequences\vspace{.15cm}

$(\ker(\lambda_M),\alpha_{\cw}|)\stackrel{\delta}\lo(\ker(\lambda_K),\alpha_{\cw}|)
\lo(\ker(\lambda_L),\alpha_{\cw}|)\stackrel{\Delta_1}\lo
(\f{M}{[M,K]},\bar\alpha_M)\lo H_1^{\alpha}(K)\twoheadrightarrow
H_1^{\alpha}(L)$,\vspace{.15cm}

$(\ker(\lambda_R),\alpha_{\cw}|)\lo(\ker(\lambda_S),\alpha_{\cw}|)
\lo(\ker(\lambda_M),\alpha_{\cw}|)\stackrel{\Delta_2}\lo
(\f{R}{[R,F]},\bar\alpha_R)$,\vspace{.15cm}\\ where $\Delta_i$,
$i=1,2$, is the connecting homomorphism of Hom-vector spaces. It
now remains to show that $\ker\Delta_2=\ker\delta$. But this
immediately follows from the commutative diagram\vspace{.4cm}

\tikzset{node distance=4.cm, auto} $~~~~$
\begin{tikzpicture}[%
back line/.style={densely dotted},
cross line/.style={preaction={draw=white, -,line width=4pt}}]
\hspace{4cm}
   \node(A1){$(\ker(\lambda_M),\alpha_{\cw}|)$};
   \node [right of=A1] (B1) {$\ds(\f{R}{[R,F]},\bar\alpha_R)$};
   \node (A2) [below of=A1, node distance=1.6cm]
         {$(\ker(\lambda_K),\alpha_{\cw}|)$};
   \node [right of=A2] (B2) {$H_2^{\alpha}(K)$,};

   \draw[->](A1) to node {$\delta$} (A2);
   \draw[<-<](B1) to node {$\subseteq$} (B2);
   \draw[->](A1) to node {$\Delta_2$} (B1);
   \draw[>->>](A2) to node {$\beta$} (B2);
\end{tikzpicture}\\
where $\beta$ is the isomorphism obtained in Theorem 4.1. Thus,
the sequence $(5)$ is exact.

We now verify the exactness of $(6)$. According to the points
mentioned at the end of Subsection 2.1, we can assume that
$(L,\alpha_L)$ is a subalgebra of $(K,\alpha_K)$ such that
$K=M\dot{+}L$. Then \vspace{-.15cm}
\[[K,K]\cap M=([L,L]+[M,K])\cap
M=([L,L]\cap M)+[M,K]=[M,K].\vspace{-.15cm}\] Also, invoking Lemma
3.2, we may consider $(M\cw K,\alpha_{\cw})$ as a subalgebra of
$(K\cw K,\alpha_{\cw})$, which implies that $(M\cw
K)\cap\ker(\lambda_K)=\ker(\lambda_M)$. So, we obtain the
following commutative diagram of Hom-Lie algebras:\vspace{.2cm}

\tikzset{node distance=3.2cm, auto}$~~~~$
\begin{tikzpicture}[%
back line/.style={densely dotted}, cross
line/.style={preaction={draw=white, -,line width=6pt}}]
\hspace{1.6cm}\node(A1){\fontsize{9.5}{5}\selectfont$~$};
   \node[right of=A1](B1){\fontsize{9.5}{5}\selectfont$\hspace{-.8cm}(\ker(\lambda_M),\alpha_{\cw}|)$};
   \node[right of=B1](C1){\fontsize{9.5}{5}\selectfont$H_2^{\alpha}(K)$};
   \node[right of=C1](D1){\fontsize{9.5}{5}\selectfont$H_2^{\alpha}(L)$};

   \node(B2)[below of=B1, node distance=1.7cm]{\fontsize{9.5}{5}\selectfont
             $\hspace{-1cm}(M\cw K,\alpha_{\cw})$};
   \node(C2)[below of=C1, node distance=1.7cm]{\fontsize{9.5}{5}\selectfont
             $(K\cw K,\alpha_{\cw})$};
   \node(D2)[below of=D1, node distance=1.7cm]{\fontsize{9.5}{5}\selectfont$(L\cw L,\alpha_{\cw})$};

   \node(B3)[below of=B2, node distance=1.7cm]{\fontsize{9.5}{5}\selectfont$\hspace{-1cm}([M,K],\alpha_{[M,K]})$};
   \node(C3)[below of=C2, node distance=1.7cm]{\fontsize{9.5}{5}\selectfont$([K,K],\alpha_{[K,K]})$};
   \node(D3)[below of=D2, node distance=1.7cm]{\fontsize{9.5}{5}\selectfont$([L,L],\alpha_{[L,L]})$,};

   \draw[>->](B1) to node{\fontsize{7.5}{5}\selectfont$~$}(C1);
   \draw[->](C1) to node{\fontsize{7.5}{5}\selectfont$\vartheta|$}(D1);

   \draw[>->](B2) to node{\fontsize{7.5}{5}\selectfont$~$}(C2);
   \draw[->>](C2) to node{\fontsize{7.5}{5}\selectfont$\vartheta$}(D2);

   \draw[>->](B3) to node{\fontsize{7.5}{5}\selectfont$~$}(C3);
   \draw[->>](C3) to node{\fontsize{7.5}{5}\selectfont$~$}(D3);

   \draw[>->](B1) to node[right]{\fontsize{7.5}{5}\selectfont$~$}(B2);
   \draw[>->](C1) to node[right]{\fontsize{7.5}{5}\selectfont$~$}(C2);
   \draw[>->](D1)to node[right]{\fontsize{7.5}{5}\selectfont$~$}(D2);

   \draw[->>](B2) to node[right]{\fontsize{7.5}{5}\selectfont$~$}(B3);
   \draw[->>](C2) to node[right]{\fontsize{7.5}{5}\selectfont$~$}(C3);
   \draw[>->>](D2)to node[right]{\fontsize{7.5}{5}\selectfont$~$}(D3);
\end{tikzpicture}\\
where the columns and rows are exact. One easily sees from the
above diagram that $\vartheta|$ is surjective.
\end{proof}
Using the above theorem and the explanations at the beginning
of this section, we have the following conjecture.\vspace{.15cm}\\
{\bf\large Conjecture.} With the assumptions of Theorem 4.6,
$H_3^\alpha(F)=0$ and there is an isomorphism of Hom-vector spaces
$H_3^{\alpha}(L)\cong(\ker(S\cw F\lo
F),\alpha_{\cw}|)$.\vspace{.15cm}

Finally, we close this section by examining the behavior of
functors $J_2^{\alpha}(-)$ and $H_2^{\alpha}(-)$ to the direct sum
of Hom-Lie algebras.
\begin{theorem}
Let $(L_1,\alpha_{L_1})$ and $(L_2,\alpha_{L_2})$ be two Hom-Lie
algebras with surjective endomorphisms $\alpha_{L_1}$ and
$\alpha_{L_2}$. Then there are isomorphisms of Hom-vector
spaces\vspace{-.1cm}
\[J_2^\alpha(L_1\oplus L_2)\cong J_2^\alpha(L_1)\oplus
J_2^\alpha(L_2)\oplus(L_1^{ab}\otimes L_2^{ab},
\alpha_{\otimes})\oplus(L_2^{ab}\otimes L_1^{ab},
\alpha_{\otimes}),\vspace{-.2cm}\]
\[\hspace{-2.95cm}H_2^{\alpha}(L_1\oplus L_2)\cong H_2^{\alpha}(L_1)\oplus H_2^{\alpha}(L_2)
\oplus(L_1^{ab}\otimes L_2^{ab},
\alpha_{\otimes}).\vspace{-.1cm}\]
\end{theorem}
\begin{proof}
we only need to prove that\vspace{-.1cm}
\begin{equation}
\hspace{.2cm}((L_1\oplus L_2)\star(L_1\oplus
L_2),\alpha_{\star})\cong(L_1\star
L_1,\alpha_{\star})\oplus(L_2\star
L_2,\alpha_{\star})\oplus(L_1^{ab}\otimes
L_2^{ab},\alpha_{\otimes}) \oplus(L_2^{ab}\otimes
L_1^{ab},\alpha_{\otimes}),\vspace{-.2cm}
\end{equation}
\begin{equation}
\hspace{-3.8cm}((L_1\oplus L_2)\cw(L_1\oplus
L_2),\alpha_{\cw})\cong(L_1\cw L_1,\alpha_{\cw})\oplus(L_2\cw
L_2,\alpha_{\cw})\oplus(L_1^{ab}\otimes
L_2^{ab},\alpha_{\otimes}).\vspace{-.1cm}
\end{equation}
Put $(L,\alpha_L)=(L_1\oplus L_2,\alpha_{\oplus})$, and identify
$(L_1,\alpha_{L_1})$ and $(L_2,\alpha_{L_2})$ with their images in
$(L,\alpha_L)$. Then $[L_1,L_2]=0$ and so, invoking Proposition
3.1$(iii)$, $(L_1\star L_2,\alpha_{\star})\cong(L_1^{ab}\otimes
L_2^{ab},\alpha_{\otimes})$ and $(L_2\star
L_1,\alpha_{\star})\cong(L_2^{ab}\otimes
L_1^{ab},\alpha_{\otimes})$. We claim that for any ideal
$(K,\alpha_K)$ of $(L,\alpha_L)$, there is an isomorphism of
Hom-Lie algebras\vspace{-.2cm}
\begin{equation}
(K\star L,\alpha_{\star})\cong((K\star L_1) \oplus(K\star
L_2),\alpha_{\oplus}).\vspace{-.2cm}
\end{equation}
Define the Hom-map $\varphi:(K\star L,\alpha_{\star})\lo((K\star
L_1)\oplus(K\star L_2),\alpha_{\oplus})$,
$(k\star(l_1,l_2))\longmapsto((k\star l_1),(k\star l_2))$. A page
of routine calculations shows that $\varphi$ preserves the
defining relations of the exterior product and also, $\varphi
\circ\alpha_{\star}=\alpha_{\oplus}\circ\varphi$. On the other
hand, the inclusions of $L_1$ and $L_2$ into $L$ yield linear
Hom-maps $(K\star L_1,\alpha_{\star})\lo (K\star
L,\alpha_{\star})$ and $(K\star L_2,\alpha_{\star})\lo (K\star
L,\alpha_{\star})$, which combine to give an inverse $\psi$ of
$\varphi$. It therefore follows that $\varphi$ is an isomorphism
of Hom-Lie algebras, as claimed. We now get the isomorphism $(7)$
by applying the isomorphism $(9)$ twice.

For the isomorphism $(8)$, it is enough to note that for the
isomorphism $\varphi$ obtained in proving $(7)$, we have
$\varphi((L_1\oplus L_2)\square(L_1\oplus L_2))=(L_1\square
L_1)\oplus(L_2\square L_2)$, which deduces the result.
\end{proof}
The following corollary generalizes a result due to Simson and Tye
\cite{ST}.
\begin{corollary}
Let $(V_1,\alpha_{V_1})$ and $(V_2,\alpha_{V_2})$ be two
Hom-vector spaces with surjective endomorphisms $\alpha_{V_1}$ and
$\alpha_{V_2}$. Then there is an isomorphism of Hom-vector
spaces\vspace{-.15cm}
\[(\Gamma(V_1\oplus V_2),\alpha_{\Gamma})\cong(\Gamma(V_1)\oplus\Gamma(V_2)
\oplus(V_1\otimes V_2),\alpha_{\oplus}).\vspace{-.15cm}\]
\end{corollary}
\begin{proof}
It follows immediately from Theorem 4.7 and Proposition 3.5.
\end{proof}
\section{The capability of Hom-Lie algebras}
In this section, we investigate some of the applications of
exterior product for developing the theory of capability of
Hom-Lie algebras. We first recall from \cite{CG1} that:\\
$\bullet$\hspace{.3cm} A Hom-Lie algebra $(L,\alpha_L)$ is said to
be {\it capable} if there exists a Hom-Lie algebra $(K,\alpha_K)$
such that $(L,\alpha_L)\cong(K/Z_{\alpha}(K),\bar\alpha_K)$. When
$\alpha_L=id_L$, this definition recovers the notion of capable
Lie algebra in \cite{SAM}.\\
$\bullet$\hspace{.3cm} Let $(L,\alpha_L)$ be any Hom-Lie algebra.

$(a)$ The {\it tensor centre} $Z^{\star}_{\alpha}(L)$ of
$(L,\alpha_L)$ is defined to be the vector space\vspace{-.2cm}
\[~~~~~~~~~Z^{\star}_{\alpha}(L)=\{x\in L~|~\alpha^k_L(x)\star l
=0_{L\star L},~{\rm for~all}~l\in L, k\geq0\}.\vspace{-.2cm}\]

$(b)$ The {\it exterior centre} $Z^{\cw}_{\alpha}(L)$ of
$(L,\alpha_L)$ is defined to be the vector space\vspace{-.2cm}
\[~~~~~~~~~~Z^{\cw}_{\alpha}(L)=\{x\in L~|~\alpha^k_L(x)\cw l
=0_{L\cw L},~{\rm for~all}~l\in L, k\geq0\}.\vspace{-.2cm}\]
Plainly, $Z^{\star}_{\alpha}(L)\subseteq Z^{\cw}_{\alpha}(L)$, and
the equality holds whenever $(L,\alpha_L)$ is perfect, by
Proposition 3.4.
\begin{example}
$(i)$ Let $(L,\alpha_L)$ be an abelian Hom-Lie algebra with linear
basis $\{e_i~|~i\in I\}$, where $I$ is a well-ordered set.
Consider the Hom-Lie algebra $(K,\alpha_K)$ with linear basis
$\{e_i, e_{jk}~|~i,j,k\in I, j<k\}$, product given by
$[e_j,e_k]=e_{jk}$ for all $j<k$, and zero elsewhere, and the
endomorphism $\alpha_K$ is defined as $\alpha_K(e_i)=\alpha_L(e_i)$
and $\alpha_K(e_{jk})=[\alpha_K(e_j),\alpha_K(e_k)]$. Since all
triple products of elements in $K$ are zero, one sees that
$(K,\alpha_K)$ satisfies the conditions of a Hom-Lie algebra.
Furthermore, $Z_{\alpha}(K)=span\{e_{jk}~|~j<k\}$, and
$(Z_{\alpha}(K),\alpha_{Z_{\alpha}(K)})\rightarrowtail(K,\alpha_K)
\stackrel{\pi}{\twoheadrightarrow}(L,\alpha_L)$ is an extension of
Hom-Lie algebras, where $\pi$ is defined by $\pi(e_i)=e_i$ and
$\pi(e_{jk})=0$. Hence $(L,\alpha_L)$ is capable.

$(ii)$ Let $(L,\alpha_L)$ be a Hom-Lie algebra with $\alpha_L=0$.
We claim that $(L,\alpha_L)$ is capable. The result is clear when
$Z_{\alpha}(L)=0$ or $(L,\alpha_L)$ is abelian. Choose a linear
basis $\{e_i~|~i\in I\}$ for $Z_{\alpha}(L)$ and extend it to a
linear basis $\{e_i, f_j~|~i\in I,j\in J\}$ for $L$, where $I$ and
$J$ are non-empty well-defined ordered sets. Take the vector space
$K$ with a linear basis $\{e_i, t_i, f_j~|~i\in I, j\in J\}$ (in
fact, $K$ is a vector space direct sum of $L$ and the space
generated by the set $\{t_i~|~i\in I\}$), together with the
following product: $[e_i,f_{j_0}]=t_i$ for some fixed $j_0\in J$
and all $i\in I$, $[f_j,f_k]$ is the same in $L$ for all $j,k\in
J$, and zero elsewhere. Then $K$ with companion endomorphism
$\alpha_K=0$ is a Hom-Lie algebra such that
$Z_\alpha(K)=span\{t_i~|~i\in I\}$ and
$(K/Z_{\alpha}(K),\bar\alpha_K)\cong(L,\alpha_L)$, as claimed.

$(iii)$ $(iii)$ Consider the three-dimensional Hom-Lie algebra
$(L,\alpha_L)$ with linear basis $\{e_1,e_2,e_3\}$, the product
given by $[e_1,e_2]=e_3$ (unwritten products are equal to zero),
and the non-surjective endomorphism $\alpha_L$ defined by
$\alpha_L(e_1)=e_3$, $\alpha_L(e_2)=e_2$, $\alpha_L(e_3)=0$.
Evidently, $Z_\alpha(L)=L^2=span\{e_3\}$. Let $(K,\alpha_K)$ be
the four-dimensional Hom-Lie algebra with linear basis
$\{f_1,f_2,f_3,f_4\}$, the product given by $[f_1,f_2]=f_3$,
$[f_3,f_1]=f_4$, and the endomorphism defined by
$\alpha_K(f_1)=f_3$, $\alpha_K(f_2)=f_2$,
$\alpha_K(f_3)=\alpha_K(f_4)=0$. Then $Z_\alpha(K)=span\{f_4\}$
and $(K/Z_{\alpha}(K),\bar\alpha_K)\cong(L,\alpha_L)$, that is,
$(L,\alpha_L)$ is capable.
\end{example}
The following proposition plays an important role in obtaining the
results of this section.
\begin{proposition} Let $(L,\alpha_L)$ be a Hom-Lie algebra. Then:

$(i)$ Both
$(Z^{\star}_{\alpha}(L),\alpha_{Z^{\star}_{\alpha}(L)})$ and
$(Z^{\cw}_{\alpha}(L),\alpha_{Z^{\cw}_{\alpha}(L)})$ are central
ideals of $(L,\alpha_L)$.

$(ii)$ $(L,\alpha_L)$ is capable if and only if
$Z^{\cw}_{\alpha}(L)=0$.

$(iii)$ If $(R,\alpha_R)\rightarrowtail(F,\alpha_F)\stackrel{\pi}
{\twoheadrightarrow}(L,\alpha_L)$ is a free presentation of
$(L,\alpha_L)$, then
$Z^{\cw}_{\alpha}(L)=\bar\pi(Z_\alpha(F/[R,F]))$, where
$\bar{\pi}:(F/[R,F],\bar{\alpha}_F)\lo(L,\alpha_L)$ is the
homomorphism induced by $\pi$.

$(iv)$ If $\alpha_L$ is a surjective endomorphism, then
$Z^{\star}_{\alpha}(L)$ is contained in $[L,L]$.

$(v)$ $(Z^{\cw}_{\alpha}(L),\alpha_{Z^{\cw}_{\alpha}(L)})$ is the
smallest central subalgebra containing all central subalgebras
$(N,\alpha_N)$ for which the canonical homomorphism
$H_2^{\alpha}(L)\lo H_2^{\alpha}(L/N)$ is a monomorphism $($or
equivalently, for which the canonical homomorphism $(L\cw
L,\alpha_{\cw})\lo(L/N\cw L/N,\alpha_{\cw})$ is an isomorphism$)$.

$(vi)$ $(Z^{\star}_{\alpha}(L),\alpha_{Z^{\star}_{\alpha}(L)})$ is
the smallest central subalgebra containing all central subalgebras
$(N,\alpha_N)$ for which the canonical homomorphism
$J_2^{\alpha}(L)\lo J_2^{\alpha}(L/N)$ is a monomorphism $($or
equivalently, for which the canonical homomorphism $(L\star
L,\alpha_{\star})\lo(L/N\star L/N,\alpha_{\star})$ is an
isomorphism$)$.
\end{proposition}
\begin{proof}
The parts $(i)$-$(iii)$ are found in \cite{CG1}.

$(iv)$ Consider the composite homomorphism $(L\star
L,\alpha_{\star})\stackrel{\rho}{\lo}(L^{ab}\star
L^{ab},\bar\alpha_{\star})\stackrel{\phi}{\lo}(L^{ab}\otimes
L^{ab},\bar\alpha_{\otimes})$, where $\rho$ is the natural
surjective homomorphism and $\phi$ is the isomorphism given in
Proposition 3.1$(iii)$. If there is some $x\in L-[L,L]$, then
$\bar x\otimes\bar y\neq 0$ for some $y\in L$, implying that the
pre-image of this element under $\phi\circ\rho$ can not be
vanished, that is, $x\star y\neq 0$ in $L\star L$. This means that
$Z^\star_\alpha(L)\subseteq[L,L]$.

$(v)$ The result immediately follows from the following
commutative diagram of Hom-Lie algebras\vspace{.1cm}

\tikzset{node distance=3cm, auto}$~~~~$
\begin{tikzpicture}[%
back line/.style={densely dotted}, cross
line/.style={preaction={draw=white, -,line width=6pt}}]
   \hspace{1cm}\node(A1){\fontsize{9.5}{5}\selectfont$~$};
   \node[right of=A1](B1){\fontsize{9.5}{5}\selectfont$\hspace{.cm}(N\cw L,\alpha_{\cw})$};
   \node[right of=B1](C1){\fontsize{9.5}{5}\selectfont$H_2^{\alpha}(L)$};
   \node[right of=C1](D1){\fontsize{9.5}{5}\selectfont$H_2^{\alpha}(\ds\f{L}{N})$};

   \node(B2)[below of=B1, node distance=1.5cm]{\fontsize{9.5}{5}\selectfont $\hspace{.cm}(N\cw
             L,\alpha_{\cw})$};
   \node(C2)[below of=C1, node distance=1.5cm]{\fontsize{9.5}{5}\selectfont $(L\cw
             L,\alpha_{\cw})$};
             \node(D2)[below of=D1, node distance=1.5cm]
            {\fontsize{9.5}{5}\selectfont$(\ds\f{L}{N}\cw\f{L}{N},\alpha_{\cw})$};

   \node(B3)[below of=B2, node distance=1.5cm]{\fontsize{9.5}{5}\selectfont $~$};
   \node(C3)[below of=C2, node distance=1.5cm]{\fontsize{9.5}{5}\selectfont
            $([L,L],\alpha_{[L,L]})$};
   \node(D3)[below of=D2, node distance=1.5cm]
     {\fontsize{9.5}{5}\selectfont$(\ds[\f{L}{N},\f{L}{N}],\alpha_{[\f{L}{N},\f{L}{N}]})$,};

   \draw[->](B1) to node{\fontsize{7.5}{5}\selectfont$~$}(C1);
   \draw[->](C1) to node{\fontsize{7.5}{5}\selectfont$~$}(D1);

   \draw[->](B2) to node{\fontsize{7.5}{5}\selectfont$~$}(C2);
   \draw[->>](C2) to node{\fontsize{7.5}{5}\selectfont$~$}(D2);

  \draw[->>](C3) to node{\fontsize{7.5}{5}\selectfont$~$}(D3);

  \draw[>->>](B1) to node[right]{\fontsize{7.5}{5}\selectfont$=$}(B2);
  \draw[>->](C1) to node[right]{\fontsize{7.5}{5}\selectfont$~$}(C2);
  \draw[>->](D1)to node[right]{\fontsize{7.5}{5}\selectfont$~$}(D2);

\draw[->>](C2) to node[right]{\fontsize{7.5}{5}\selectfont$~$}(C3);
\draw[->>](D2)to node[right]{\fontsize{7.5}{5}\selectfont$~$}(D3);
\end{tikzpicture}\\
where, by the sequence $(2)$ and Theorem 4.1, the rows and columns
are exact.

The proof of $(vi)$ is similar to that of $(v)$.
\end{proof}
We have the following consequences, which are of interest in their
own account.
\begin{corollary}
Let $(N,\alpha_N)$ be a central subalgebra of a finite dimensional
Hom-Lie algebra $(L,\alpha_L)$. Then $N\subseteq
Z^{\cw}_{\alpha}(L)$ if and only if
$\dim(H_2^{\alpha}(L/N))=\dim(H_2^{\alpha}(L))+\dim(N\cap[L,L])$
$($as vector spaces$)$.
\end{corollary}
\begin{proof}
The centrality of $(N,\alpha_N)$ together with Theorem 4.6 implies
an exact sequence\vspace{-.2cm}
\[(N\cw K,\alpha_{\cw})\lo H_2^{\alpha}(\f{L}{N})\lo
H_2^{\alpha}(L)\twoheadrightarrow(N\cap[L,L],\alpha_{N\cap[L,L]}).\vspace{-.2cm}\]The
result now follows from Proposition 5.2$(v)$.
\end{proof}
\begin{corollary}
For any Hom-Lie algebra $(L,\alpha_L)$ with surjective
endomorphism $\alpha_L$,\vspace{-.2cm}
\[Z^{\star}_{\alpha}(L)=Z^{\cw}_{\alpha}(L)\cap[L,L].\vspace{-.2cm}\]
\end{corollary}
\begin{proof}
Let $x\in Z^{\cw}_{\alpha}(L)\cap[L,L]$. Then
$\alpha^k(x)\in[L,L]$ and $\alpha^k(x)\star l\in L\square L$ for
all $l\in L$ and $k\geq0$, implying from Corollary 3.6 that
$\alpha^k(x)\star l=0$. Hence $x\in Z^{\star}_{\alpha}(L)$. The
inverse containment follows from Proposition 5.2$(iv)$ and the
fact that $Z^{\star}_{\alpha}(L)\subseteq Z^{\cw}_{\alpha}(L)$.
\end{proof}
If $\alpha_L=id_L$, Corollaries 5.3 and 5.4 reduce to
\cite[Theorem 4.4]{SAM} and \cite[Corollary 2.7]{N}, respectively.
\begin{corollary}
Let $(L,\alpha_L)$ be a perfect Hom-Lie algebra with surjective
endomorphism $\alpha_L$. Then $(L,\alpha_L)$ is capable if and
only if $Z(L)=0$.
\end{corollary}
\begin{proof}
The surjectivity of $\alpha_L$ yields that $Z_\alpha(L)=Z(L)$. We
have $(Z(L)\star L,\alpha_{\star})\cong(Z(L)\otimes
L^{ab},\alpha_{\otimes})=0$, thanks to Proposition 3.1$(iii)$.
Consequently $(Z(L)\cw L,\alpha_{\cw})=0$, forcing the natural
homomorphism $H_2^{\alpha}(L)\lo H_2^{\alpha}(L/Z(L))$ is
injective. Therefore, the parts $(ii)$ and $(v)$ of Proposition
5.2 imply the result.
\end{proof}
In the following, we try to omit the condition of surjectivity in
Corollary 5.5.
\begin{proposition}
A perfect Hom-Lie algebra $(L,\alpha_L)$ is capable if and only if
$Z_{\alpha}(L)\subseteq\ker(\alpha_L)$.
\end{proposition}
\begin{proof}
We can assume $Z_\alpha(L)\neq0$. If there is $x\in Z_{\alpha}(L)$
such that $\alpha_L(x)\neq0$, then $\alpha_L^k(x)\cw[a,b]=0$ for
all $a,b\in L$ and $k\geq 1$. The perfectness of $(L,\alpha_L)$
yields that $\alpha_L(x)\in Z_\alpha^\cw(L)$, and then
$(L,\alpha_L)$ is not capable, thanks to Proposition 5.2(ii). We
now suppose $\alpha_L(Z_{\alpha}(L))=0$ and consider the algebra
$K$ introduced in Example 5.1$(ii)$. We define the endomorphism
$\alpha_K$ such that $\alpha_K|_{L}=\alpha_L$ and
$\alpha_K|_{span\{t_i|i\in I\}}=0$. It is straightforward to check
that $(K,\alpha_K)$ is a Hom-Lie algebra with
$Z_{\alpha}(L)=span\{t_i~|~i\in I\}$ and
$(K/Z_{\alpha}(K),\bar\alpha_K)\cong(L,\alpha_L)$.
\end{proof}
The following theorem shows that for deciding on the property of
capability, we may restrict ourselves to Hom-Lie algebras with
surjective endomorphisms.
\begin{theorem}
Let $(L,\alpha_L)$ be a non-perfect Hom-Lie algebra with
non-surjective endomorphism. Then $(L,\alpha_L)$ is capable.
\end{theorem}
\begin{proof}
Owing to Example 5.1$(i)$, we can assume that $(L,\alpha_L)$ is
non-abelian. If $L-[L,L]\subseteq\Im(\alpha_L)$, then
$L=[L,L]\cup\Im(\alpha_L)$. But this implies that $(L,\alpha_L)$
is perfect or $\alpha_L$ is surjective, contradicting the
assumptions. So, there is an element $x\in
L-([L,L]\cup\Im(\alpha_L))$ with
$x\in$\hspace{-.2cm}$/~Z_\alpha(L)$. Choose a linear basis
$\{e_i~|~i\in I\}$ for $Z_{\alpha}(L)$ and extend it to a linear
basis $\{e_i, f_j~|~i\in I,j\in J\}$ for $L$ containing $x$, where
$I$ and $J$ are non-empty sets. Take the algebra $K=L\dot{+}T$,
where $T$ is generated by the set $\{t_i~|~i\in I\}$, together
with the following product: $[e_i,x]=t_i$ for all $i\in I$,
$[f_j,f_k]$ is the same in $L$ for all $j,k\in J$, and zero
elsewhere. Define $\alpha_K$ such that $\alpha_K|_L=\alpha_L$ and
$\alpha_K(t_i)=0$ for $i\in I$. Since $x\not\in\Im(\alpha_K)$, $x$
does not appear in the brackets of the form $[\alpha_K(a),b]$,
where $a\in K$, $b\in[K,K]$. Hence, for all $i\in I$, the elements
$t_i$ are never vanished by force the Hom-Jacobi identity. On the
other hand, we have $[K,t_i]=0$ for any $i\in I$. Therefore,
$(K,\alpha_K)$ is a Hom-Lie algebra,
$Z_{\alpha}(K)=span\{t_i~|~i\in I\}$ and
$(K/Z_{\alpha}(K),\bar\alpha_K)\cong(L,\alpha_L)$.
\end{proof}
It is obvious that if the Hom-Lie algebras $(L_1,\alpha_{L_1})$
and $(L_2,\alpha_{L_2})$ are capable, then $(L_1\oplus
L_2,\alpha_{\oplus})$ is capable, because
$Z^{\cw}_{\alpha}(L_1\oplus L_2)\subseteq
Z^{\cw}_{\alpha}(L_1)\oplus Z^{\cw}_{\alpha}(L_2)$. In the
following, we prove the converse of this result, under some
conditions.
\begin{theorem}
Let $(L,\alpha_L)$ be a finite dimensional capable regular Hom-Lie
algebra such that $(L,\alpha_L)=(L_1\oplus L_2,\alpha_{\oplus})$.
Then $(L_1,\alpha_{L_1})$ and $(L_2,\alpha_{L_2})$ are capable.
\end{theorem}
\begin{proof}
We only need to prove that $Z^{\cw}_{\alpha}(L_1\oplus
L_2)=Z^{\cw}_{\alpha}(L_1)\oplus Z^{\cw}_{\alpha}(L_2)$. By
Example 5.1$(i)$, we can assume that $(L,\alpha_L)$ is
non-abelian. For convenience, we divide the rest of the proof into
three steps.

{\it Stem} $1$. Here we prove that if $(L_i,\alpha_{L_i})$ is
non-abelian, then $(L_i,\alpha_{L_i})=(T_i\oplus
A_i,\alpha_{\oplus})$ such that $(A_i,\alpha_{A_i})$ is an abelian
Hom-Lie algebra and $Z^{\cw}_{\alpha}(L_i)=Z^{\cw}_{\alpha}(T_i)$,
for $i=1,2$.

It follows from \cite[Theorem 3.8]{PNP} that
$(L_i,\alpha_{L_i})=(T_i\oplus A_i,\alpha_{\oplus})$ in which
$(A_i,\alpha_{A_i})$ is abelian and $[L_i,L_i]\cap
Z(L_i)=Z(T_i)\subseteq[T_i,T_i]$. By Example 5.1$(i)$,
$(A_i,\alpha_{A_i})$ is capable and then
$Z^{\cw}_{\alpha}(A_i)=0$, implying that
$Z^{\cw}_{\alpha}(L_i)\subseteq Z^{\cw}_{\alpha}(T_i)$. We now
claim that $Z^{\cw}_{\alpha}(T_i)\subseteq Z^{\cw}_{\alpha}(L_i)$.
By virtue of Theorem 4.7 and Corollary 5.3, we have\vspace{-.2cm}
\begin{alignat*}{1}
\dim(H_2^{\alpha}(L_i))&=\dim(H_2^{\alpha}(T_i))+
\dim(H_2^{\alpha}(A_i))+\dim((T_i^{ab}\otimes A_i)\oplus
(A_i\otimes T_i^{ab}))\\
&=\dim(H_2^{\alpha}(\f{T_i}{Z^{\cw}_{\alpha}(T_i)}))-\dim(Z^{\cw}_{\alpha}(T_i))+
\dim(H_2^{\alpha}(A_i))\\
&+\dim((T_i^{ab}\otimes A_i)\oplus
(A_i\otimes T_i^{ab}))\\
&=\dim(H_2^{\alpha}(\f{L_i}{Z^{\cw}_{\alpha}(T_i)}))-\dim(Z^{\cw}_{\alpha}(T_i))
\end{alignat*}
which, using again Corollary 5.3, gives the required result.

{\it Stem} $2$. Here we prove that if the Hom-Leibniz algebras
$(L_i,\alpha_{L_i})$, $i=1,2$, are both non-abelian and
$Z(L_i)\subseteq[L_i,L_i]$, then $Z^{\cw}_{\alpha}(L)=
Z^{\cw}_{\alpha}(L_1)\oplus Z^{\cw}_{\alpha}(L_2)$.

It is enough to verify that $Z^{\cw}_{\alpha}(L_i)\subseteq
Z^{\cw}_{\alpha}(L)$. Using Theorem 4.7 and an argument similar to
Step $1$, we deduce that\vspace{-.2cm}
\[\dim(H_2^{\alpha}(L))=\dim(H_2^{\alpha}(\f{L}{Z^{\cw}_{\alpha}(L_i)}))
-\dim(Z^{\cw}_{\alpha}(L_i))\vspace{-.1cm},\]from which we have $
Z^{\cw}_{\alpha}(L_i)\subseteq Z^{\cw}_{\alpha}(L)$.

{\it Stem} $3$. The completion of the proof.

If the one of Hom-Lie algebras $(L_1,\alpha_{L_1})$ or
$(L_2,\alpha_{L_2})$ is abelian, the result obtains from Step $1$.
So suppose that both are non-abelian. Then, using again Step $1$,
there are non-abelian-subalgebras $(T_i,\alpha_{T_i})$ of
$(L_i,\alpha_{L_i})$, $i=1,2$, such that
$(L,\alpha_{L})=(T_1\oplus T_2\oplus A,\alpha_{\oplus})$,
$Z^{\cw}_{\alpha}(L)=Z^{\cw}_{\alpha}(T_1)\oplus
Z^{\cw}_{\alpha}(T_2)$,
$Z^{\cw}_{\alpha}(L_i)=Z^{\cw}_{\alpha}(T_i)$ and
$Z(T_i)\subseteq[T_i,T_i]$. But by Step $2$,
$Z^{\cw}_{\alpha}(T_1\oplus T_2)=Z^{\cw}_{\alpha}(T_1)\oplus
Z^{\cw}_{\alpha}(T_2)$. We therefore conclude that
$Z^{\cw}_{\alpha}(L_1\oplus L_2)=Z^{\cw}_{\alpha}(L_1)\oplus
Z^{\cw}_{\alpha}(L_2)$, as desired.
\end{proof}
The following example shows that the regular condition is
essential in Theorem 5.8.
\begin{example}
Let $H(m)$ be the Heisenberg algebra of dimension $2m+1$. Then by
\cite[Theorem 3.4]{NPR}, the Hom-Lie algebra $(H(m),id_{H(m)})$ is
not capable for any $m\geq2$. But if we define $K=span\{t\}$ and
$\alpha_K(t)=0$, then thanks to Theorem 5.7, we infer that
$(H(m),id_{H(m)})\oplus (K,\alpha_K)$ is capable.
\end{example}


\begin{thebibliography}{0}
{\small
\bibitem{AEM} Ammar, F., Z. Ejbehi, A. Makhlouf, \emph{Cohomology and deformations of
Hom-algebras}, J. Lie Theory \textbf{21}(2011),
813-836.\vspace{-.3cm}
\bibitem{ASR} Aslizadeh, A., A.R. salemkar, Z. Riyahi, \emph{Polynilpotent capability of Lie
algebras}, Comm. Algebra \textbf{47}(2019),
1390-1400.\vspace{-.3cm}
\bibitem{BB} Borceux, F., D. Bourn, \emph{Mal'cev, Protomodular, Homological and Semi-Abelian
Categories},  Mathematics and Its Applications, vol. \textbf{566},
Kluwer, Dordrecht (2004).\vspace{-.3cm}
\bibitem{CG} Casas, J.M., X. Garcia-Martinez, \emph{Abelian extensions and crossed modules
of Hom-Lie algebas}, J. Pure Appl. Algebra \textbf{224}(2020),
987-1008.\vspace{-.3cm}
\bibitem{CG1} Casas, J.M., X. Garcia-Martinez, \emph{On the
capability of Hom-Lie algebas}, arXiv: 2101.11522v1,
2021.\vspace{-.3cm}
\bibitem{CIP} Casas, J.M., M.A.Insua, N. Pacheco, \emph{On universal central
extensions of Hom-Lie algebras}, Hacet. J. Math. Stat.
\textbf{44}(2015), 277-288.\vspace{-.3cm}
\bibitem{CKR} Casas, J.M., E. Khmaladze, N.P. Rego, \emph{A non-abelian tensor product
of Hom-Lie algebras}, Bull. Malays. Math. Sci. Soc.
\textbf{40}(2017), 9144-9163.\vspace{-.3cm}
\bibitem{CS} Cheng, Y., Y. Su, \emph{$($Co$)$homology and
universal central extension of Hom-Leibniz algebras}, Acta Math.
Sin. $($Engl. Ser.$)$ \textbf{27}(2011), 813-830.\vspace{-.3cm}
\bibitem{EHS} Edalatzadeh, B., S.N. Hosseini, A.R. Salemkar,
\emph{On characterizing pairs of nilpotent Lie algebras by their
second relative homologies}, J. Algebra \textbf{549} (2020),
112-127.\vspace{-.3cm}
\bibitem{E} Ellis, G., \emph{Nonabelian exterior products of Lie algebras and
an exact sequence in the homology of Lie algebras}, J. Pure Appl.
Algebra \textbf{46}(1987), 111-115.\vspace{-.3cm}
\bibitem{E1} Ellis, G., \emph{A non-abelian tensor product of
Lie algebras}, Glasgow Math. J. \textbf{39}(1991),
101-120.\vspace{-.3cm}
\bibitem{EV} Everaert, T., T. Van der Linden, \emph{Baer
invariants in semi-abelian categories. I. General theory}, Theory
Appl. Categ. \textbf{12}(2004), 1-33.\vspace{-.3cm}
\bibitem{HLS} Hartwig, J.T., D. Larsson, S.D. Silvestrov, \emph{Deformations of Lie
algebras using $\sigma$-derivations}, J. Algebra
\textbf{295}(2006), 314-361.\vspace{-.3cm}
\bibitem{H} Higgins, P.J., \emph{Groups with multiple operators}, Proc. London Math. Soc.
\textbf{(3)6}(1956), 366-416.\vspace{-.3cm}
\bibitem{IKL} Inassaridze, N., E. Khmaladze, M. Ladra, \emph{Non-abelian tensor
product of Lie algebras and its derived functors}, Extracta Math.
\textbf{17}(2002), 281-288.\vspace{-.3cm}
\bibitem{K} Khmaladze, E., \emph{Non-abelian tensor and exterior products modulo q and
universal q-central relative extension of Lie algebras}, Homology,
Homotopy and Applications \textbf{1}(1999), 187-204.\vspace{-.3cm}
\bibitem{K1} Khmaladze, E., \emph{Homology of Lie algebras with $\Lambda/q\Lambda$
coefficients and exact sequences}, Theory and Applications of
Categories \textbf{10}(2002), 113-126.\vspace{-.3cm}
\bibitem{JL} Jin, Q., X. Li, \emph{Hom-Lie algebra structures on semi-simple Lie
algebras}, J. Algebra \textbf{319}(2008), 1398-1408.\vspace{-.3cm}
\bibitem{LS} Larsson, D., S.D. Silvestrov, \emph{Quasi-Hom-Lie
algebras, central extensions and $2$-cocycle-like identities}, J.
Algebra \textbf{288}(2005), 321-344.\vspace{-.3cm}
\bibitem{N} Niroomand, P., \emph{Some properties on the tensor square of Lie
algebras}, J. Algebra Appl. \textbf{11}(2012),
1250085.\vspace{-.3cm}
\bibitem{NJP} Niroomand, P., F. Johari, M. Parvizi, \emph{On the capability
and Schur multiplier of nilpotent Lie algebra of class two}, Proc.
Amer. Math. Soc. \textbf{144} (2016), 4157-4168.\vspace{-.3cm}
\bibitem{NPR} Niroomand, P., M. Parvizi, F.G. Russo,
\emph{Some criteria for detecting capable Lie algebras}, J.
Algebra \textbf{384}(2013), 36-44.\vspace{-.3cm}
\bibitem{PNP} Padhan, R.N., N. Nandi, K.C. Pati, \emph{Some properties of factor set in regular
Hom-Lie algebras}, arXiv: 2005.05555v1, 2020.\vspace{-.3cm}
\bibitem{SE} Salemkar, A.R., B. Edalatzadeh, \emph{The multiplier and the cover
of direct sums of Lie algebras}, Asian-Eur. J. Math.
\textbf{5}(2012), 1250026(1-6).\vspace{-.3cm}
\bibitem{SEA} Salemkar, A.R., B. Edalatzadeh, M. Araskhan, \emph{Some inequalities for
the dimension of the $c$-nilpotent multiplier of Lie algebras}, J.
Algebra \textbf{322}(2009), 1575-1585.\vspace{-.3cm}
\bibitem{SAM} Salemkar, A.R., V. Alamian, H. Mohammadzadeh, \emph{Some properties of the Schur
multiplier and covers of Lie algebras}, Comm. Algebra
\textbf{36}(2008), 697-707.\vspace{-.3cm}
\bibitem{S} Sheng, Y., \emph{Representations of Hom-Lie algebras},
Algebra Represent. Theory} \textbf {15}(2012),
1081-1098.\vspace{-.3cm}
\bibitem{ST} Simson, D., A. Tye, \emph{Connected sequences of stable derived
functors and their applications}, Dissertationes Mathematicae
(Rozprawy Mat.) \textbf{111}(1974).\vspace{-.3cm}
\bibitem{Y} Yau, D., \emph{Hom-algebras and homology}, J. Lie
Theorey \textbf{19}(2009), 409-421.\vspace{-.3cm}
\bibitem{Y1} Yau, D., \emph{The Hom-Yang-Baxter equation, Hom-Lie
algebras and quasi-triangular bialgebras}, J. Phys. A \textbf
{42}(2009), 165202 (12pp).\vspace{-.3cm}
\bibitem{Y2} Yau, D., \emph{The Hom-Yang-Baxter equation and Hom-Lie
algebras}, arXiv:0905.1887.\vspace{-.3cm}
\end{thebibliography}
\end{document}